\documentclass{article}
\usepackage{amsmath}
\usepackage{amssymb}
\usepackage{titlesec}
\usepackage{graphicx}
\usepackage{fancyhdr}
\usepackage{latexsym}
\usepackage{mathrsfs}
\usepackage{booktabs}
\usepackage{mathrsfs}
\usepackage{mathtools}
\usepackage{amsthm}
\usepackage{wasysym}
\usepackage{cite}
\usepackage{color}
\usepackage{ulem}
\usepackage{todonotes}
\usepackage{amsfonts}
\usepackage{graphics}
\usepackage{multimedia}
\usepackage{amscd}
\numberwithin{equation}{section}

\newtheorem{theorem}{Theorem}[section]
\newtheorem{lemma}{Lemma}[section]

\newtheorem{proposition}{Proposition}[section]

\newcommand{\dv}{\text{div}}

\title{Global strong solution of the 3D compressible liquid
crystal flows with density-dependent viscosity and large
velocity}
\author{Jiaxu L{\small I}$^{b}$, Yu M{\small EI}$^c$, Rong Z{\small HANG}$^{a,d}$ \thanks{Email addresses:  Jiaxvlee@gmail.com (J. X. Li), yu.mei@nwpu.edu.cn (Y. Mei), rzhang0921@gmail.com (R. Zhang). }  \\ 
{\normalsize a. School of Mathematics and Computer Sciences,}\\
{\normalsize Nanchang University, Nanchang 330031, P. R. China;}\\
{\normalsize b.  School of Mathematical Sciences,}\\
{\normalsize  Shenzhen University, Shenzhen 518060, P. R. China;}\\
{\normalsize c. School of Mathematics and Statistics,}\\
{\normalsize Northwestern Polytechnical University, Xi'an Shaanxi, 710129, P. R. China}\\
{\normalsize d. Institute of Mathematics and Interdisciplinary Sciences,}\\
{\normalsize Nanchang University, Nanchang 330031, P. R. China.}
}
\date{}

\begin{document}
\maketitle
\begin{abstract}
This paper concerns the Cauchy problem of three-dimensional compressible liquid crystal flows with density-dependent viscosity. When the viscosity coefficients $\mu_1(\rho),\mu_2(\rho)$ are power functions of the density with the power larger than $1$, it is proved that the system exists a unique global strong solution as long as the initial density is sufficiently large and $L^3$-norm of the derivative of the initial director is sufficiently small.  This is the first result concerning the global strong solution for three-dimensional compressible liquid crystal flows without smallness of velocity.

\end{abstract}

\section{Introduction}

In this paper, we concern the following simplified Ericksen-Leslie system for compressible liquid crystal flows in three-dimensional whole space $\mathbb{R}^3$ 
\begin{equation}\label{ins}
	\begin{cases}
	\rho_t+\mathrm{div}(\rho u)=0,\\
		(\rho u)_t+\mathrm{div}(\rho u\otimes u)+\nabla P-\mathrm{div} \mathbb{T}=-\nu\mathrm{div}(\nabla d\odot\nabla d-\frac{1}{2} |\nabla d|^2 \mathbb{I}_3), \\
        d_t+u\cdot\nabla d=\lambda(\Delta d+|\nabla d|^2d),\\
        |d|=1,
	\end{cases} 
\end{equation}
% for liquid crystal flows in a three-dimensional domain $\Omega\subset\mathbb{R}^3.$ 
The unknowns $\rho, u$ and $d$ represent, respectively, the density, the velocity and the direction field. The pressure $P$ is given by
\begin{equation}\label{eqn-of-P}
P=a\rho^\gamma,
\end{equation}
where $a>0$ is an entropy constant and $\gamma>1$ is the adiabatic exponent. The viscous stress tensor $\mathbb{T}$ has the following form
\begin{equation}\label{eqn-of-T}
\mathbb{T}=2 \mu_1(\rho) \mathcal{D}u +\mu_2(\rho)\mathrm{div}u\mathbb{I}_3,
\end{equation}
where $\mathcal{D}u =\frac{1}{2}\left(\nabla u+(\nabla u)^\top \right)$ is the deformation tensor and $\mathbb{I}_3$ is the $3\times 3$ identity matrix, and the viscosities $\mu_1(\rho)$ and $\mu_2(\rho)$ are assumed to satisfy 
\begin{equation}\label{vis-d}
\mu_1(\rho)=\mu_1\rho^\alpha,\ \mu_2(\rho)=\mu_2\rho^\alpha,
\end{equation}
with constants $\alpha\geq 0$ and
\begin{equation}\label{pc}
\mu_1>0,\ 2\mu_1+3\mu_2\geq 0.
\end{equation}
The system \eqref{ins} is supplemented with the following initial conditions
\begin{equation}\label{ini-data}
\rho(0,x)=\rho_0(x),~~u(0,x)=u_0(x),~~d(0,x)=d_0(x), ~~~~x\in\mathbb{R}^3,
\end{equation}
and  the far-field behaviour
\begin{equation}\label{bc}
    (\rho, u, d)(t,x) \to (\bar \rho ,0, e), \quad  \text{as}\  |x|\to \infty, t\ge0, 
\end{equation}
with $e$ being a given unit vector.

The system \eqref{ins} is the simplified hydrodynamic model of liquid crystal flows introduced by Ericksen\cite{ericksen1961conservation} and Leslie\cite{leslie1968some} in 1960s. It is a coupling system between compressible Navier-Stokes system and the harmonic map heat flow from $\mathbb{R}^3$ to $\mathbb{S}^2$. In fact, \eqref{ins} reduce to the compressible isentropic Navier-Stokes system if $d$ is a constant vector, while the harmonic heat flow if $\rho$ a constant and $u$ zero. For the compressible isentropic Navier-Stokes system, there is a lot of literature on the well-posedness of solutions to the Cauchy problem in multi-dimensional space. When viscosity coefficients are positive constants $(\alpha=0)$, Lions proved, in the pioneer works \cite{Lions1993,Lions1998-2}, the global existence of weak solutions  for large initial data provided $\gamma\geq\frac{9}{5}$, which is improved later by Feireisl, Novotny and Petzeltov\'a \cite{Feireisl2001} for $\gamma>\frac{3}{2}$ and Jiang and Zhang\cite{JiangZhang} for the spherically symmetric weak solution with $\gamma>1$. Concerning global classical solutions, Matsumura and Nishida \cite{Matsumura1980} first obtained the well-posedness of classical solutions for initial data close to a non-vacuum equilibrium in the Sobolev space $H^s$. Later, Huang, Li and Xin \cite{HuangLiXin2012} obtained the global existence of classical solutions to the Cauchy problem in three dimension with vacuum and large oscillations provided that the initial energy is suitably small. Li and Xin \cite{LiXin2019} further obtained some large time decay rates of the pressure and the spatial gradient of the velocity field provided that the initial total energy is suitably small, and shown the global well-posedness and large time asymptotic behavior of
strong and classical solutions to the Cauchy problem of compressible Navier–Stokes equations
in two or three spatial dimensions with
far field vacuum. When viscosity coefficients depend on the density, Li and Xin \cite{15li-xin}, Vasseur and Yu \cite{Vasseur2016} have independently proved the global existence of weak solutions for compressible Navier-Stokes systems with the two viscosity coefficients being power functions of the density and satisfying the B-D entropy condition \cite{bresch2004some}. The extended result for non-linear density-dependent viscosities satisfying the B-D entropy condition is obtained in \cite{bresch2021global}. For global strong or classical solutions, Kazhikhov and Vaigant \cite{kazhikhov1995existence} introduced density-dependent viscosities with constant $\mu_1$ and $\mu_2(\rho)=\rho^\beta$ and established the global existence of strong solutions to the two-dimensional problem with large initial data away from the vacuum provided $\beta>3$. Recently, this result is generalized in series of work \cite{JiuWangXin2013,JiuWangXin2018,huang2022global,fan2022global}from $\beta>3$ to $\beta>\frac{4}{3}$ for the Cauchy and initial boundary value problems.  For density-dependent viscosities as in \eqref{vis-d}, Xin and Zhu \cite{XIN2021108072} obtained the global existence of regular solutions with vacuum at infinity if $\alpha>1$ and $\alpha\geq2\gamma-1$ for a class of smooth initial data that are of small density but possibly large velocities. Very recently, Huang, Li and Zhang\cite{HuangLiZhang2024} proved the global well-posedness of strong solutions if $\alpha,\gamma>1$ and $\alpha>\max\{\frac{\gamma+1}{2},\gamma-1\}$ and the initial density sufficiently large, which extend the range of $\alpha,\gamma$ required in \cite{yu2023global}. We refer to \cite{Cho2004unique,LiLiang,17jmfm-li-pan-zhu,19arma-li-pan-zhu,zhuxin2021}  for the related local well-posedness results as well as \cite{CaoLiZhu,WenZhang} and reference therein for the one-dimensional problem. 

For the harmonic map heat flows, Eells and Sampson, in their seminal work \cite{EellsSampson}, showed that there exists a globally smooth solution for smooth data if the sectional
curvature of target manifold is non-positive. Without the curvature assumption, Struwe\cite{Struwe1985} proved the existence of a global weak solutions with non-increasing Dirichlet
energy to the harmonic map heat flow in Riemannian surface, which is smooth away from finitely
many possible singular points. Chen and Struwe\cite{ChenStruwe} showed the existence of a global, partially regular weak solution
of the harmonic map heat flow in higher dimensions with the help of the monotonicity formula discovered by Struwe\cite{Struwe1988} and the Ginzburg-Landau type approximation scheme introduced.  The uniqueness was proved by Freire\cite{Freire1995} in the Struwe-type class obtained in \cite{Struwe1985}. However,  Coron\cite{Coron} constructed an example to shown that the harmonic map heat flow in higher dimensions may fail to
be unique among Chen-Struwe solutions in\cite{ChenStruwe}.  It was also shown that the Struwe-type solution may indeed blow up in finite time. The first example was found by Chang, Ding and Ye\cite{Chang} in the radially symmetric class from the disk to the sphere. Very recently, D\'{a}vila, del Pino and Wei\cite{Davila} constructed the example of a solution, which blows up precisely at any given finite
points in a two-dimensional bounded, smooth domain with suitable initial and boundary values. 

For the simplified compressible liquid crystal system \eqref{ins} with constant viscosity coefficients, global existence of weak solutions with finite energy were obtained in \cite{JiangJiangWang2013,JiangJiangWang2014,LinLaiWang} under the geometric condition that the initial orientational director field $d_0$ lies in the hemisphere. The geometric condition as well as the maximum principle for the director are key ingredients for global compactness from the Ginzburg-Landau penalty term $\frac{1}{\varepsilon^2}(1-|d|^2)d$ to the supercritial nonlinear term $|\nabla d|^2d$, which is firstly established for incompessible liquid crystal flows in \cite{lin2016global}. However, it still remains open to remove this geometric condition and prove the global weak solutions with finite energy even for the simplified incompressible liquid crystal system in $\mathbb{R}^3$. 
As to global strong and classical solutions, there are some well-posedness results under some smallness assumption on the initial data. For initial data being close to the constant equilibrium
state $(1,0,e)$, Hu and Wu proved the global existence and uniqueness of strong solutions in critical Besov spaces,  while Gao, Tao and Yao\cite{gaotaoyao} obtained the global existence and long-time behavior of classical solution in $H^s(\mathbb{R}^3)$($s\geq 3$).  Li, Xu and Zhang\cite{LiXuZhang} established the global existence of classical solutions with large oscillations and vacuum to three-dimensional compressible liquid crystal flows. The results for two-dimensional Cauchy problem and initial-boundary one were obtained by Wang \cite{Wangt2016}, Sun and Zhong\cite{SunZhong} respectively. For the density-dependent viscosity case, according to our acknowledge, the only result was obtained in \cite{ZhongZhou2024}, where they proved the global well-posedness of strong solutions to the two-dimensional Cauchy problem for the Kazhikhov-Vaigant type viscosity with $\beta>\frac{4}{3}$.

In this paper, we will investigate the global existence and uniqueness of strong solutions to the Cauchy problem to the system \eqref{ins} with density-depend viscosity coefficients of the form in \eqref{vis-d}. We would like to establish the global existence of strong solutions to the system \eqref{ins}-\eqref{vis-d} provided that the density is large enough and as well as the $L^3$norm of the derivative of director is small enough under some mild assumptions on $\alpha$. This is motivated by the corresponding results for compressible Navier-Stokes system obtained by Huang, and the first and last authors in \cite{HuangLiZhang2024} and inhomogenous liquid crystal flows in \cite{LiMeiZhang}. This result is based on the observation that the Reynolds number $Re=\rho uL/\mu_i(\rho)$ will be small provided $\alpha > 1$ and the density $\rho$ is large enough, no matter how fast the flow speed is, so that it does not required any smallness assumption of velocity. However, being subject to our approach, it seems impossible to further remove the smallness of director due to supercritical nonlinearity term $|\nabla d|^2d$ and nonlinear constrain $|d|=1$ in the director equation. As a result, we impose the smallness of director in critical space in the sense that the system \eqref{ins} is invariant under the scaling transform
$$\rho_\tau(x,t)=\rho(\tau x,\tau^2t),\quad u_\tau(x,t)=u(\tau x,\tau^2t),\quad d_\tau(x,t)=d(\tau x,\tau^2t).$$

Before stating the main results, we explain the notation and conventions used throughout this paper. 
Denote
\begin{equation*}
	\int f \mathrm{dx}=\int_{\mathbb{R}^3} f\mathrm{dx}.
\end{equation*}
Define
    \begin{equation}\label{g}
        G(\rho)=\rho \int_{\bar \rho }^{\rho} \frac{P(s)-P(\bar\rho)}{s^2}ds.
    \end{equation}
For $1\le r\le \infty$ and integer $k\ge 0$, we denote the standard homogeneous and inhomogeneous Sobolev spaces as follows:
\begin{equation*}
\begin{gathered}
L^r=L^r(\mathbb{R}^3), \quad D^{k,r}=\{u\in L^1_{loc}(\mathbb{R}^3)\ |\  \|\nabla^k u\|_{L^r}<\infty\},\\
W^{k,r}=L^r \cap D^{k,r}, \quad H^k=W^{k,2}.
\end{gathered}
\end{equation*}

Our main result is stated as follows.
\begin{theorem}\label{global} 
    Let $\bar \rho>1$ be a given constant and $\alpha,\gamma$ satisfy that
\begin{equation}\label{vis-c}
        \alpha>1, \gamma>1, \alpha>\frac{\gamma+1}{2}, \alpha\ge \gamma-1.
    \end{equation}
    Suppose that the initial data $(\rho_0,u_0,d_0)$ satisfies 
\begin{equation}\label{ia}
		\frac{3}{4} \bar \rho \le \rho_0\le \frac{5}{4} \bar \rho,\quad \rho_0-\bar \rho \in D^{1,2} \cap %D^{1,2} \cap 
        D^{1,q},\ 3<q<6, G(\rho_0)\in L^1, \quad  u_0 \in H^2,\quad d_0-e\in H^3.
    \end{equation}
    Then there exists positive constants $\Lambda_0$ depending only
	on $a, \mu_1,\mu_2, \nu, \lambda, \alpha,$ $\|G(\rho_0)\|_{L^1}, \|\nabla\rho_0\|_{L^2},$ $ \|\nabla\rho_0\|_{L^q}, $ $\| u_0\|_{H^2},$ and $ \|\nabla d_0\|_{H^2},$ and $ \varepsilon_0$ depending only on Sobolev constants such that if
	\begin{equation}\label{ini}
		\bar\rho\ge \Lambda_0, \quad \|\nabla d_0\|_{L^3} \le \varepsilon_0,
	\end{equation}
    then the compressible simplified Ericksen-Leslie system  (\ref{ins})-(\ref{bc}) admits a unique global strong solution $(\rho,u, d)$ in $\mathbb{R}^3\times(0,\infty)$ satisfying \begin{equation}
		\frac{2}{3} \bar \rho \le \rho \le \frac{4}{3} \bar \rho,
    \end{equation}
    and 
    \begin{equation}
    \left\{ \begin{array}{l}
    \rho \in C([0,\infty);D^{1,2}\cap D^{1,q}),\
    \nabla u\in C([0,\infty);H^1)\cap L^2((0,\infty);W^{1,q}),\\
    \rho_t \in C([0,\infty);L^q),\
    \sqrt{\rho}u_t\in L^\infty((0,\infty);L^2),\ 
    u_t\in L^2((0,\infty);H^1_0),\\
    d-e\in C([0,T;H^3])\cap L^2(0,T;H^4), d_t\in C([0,T;H^1])\cap L^2(0,T;H^2).
    \end{array} \right.
\end{equation}
\end{theorem}

\section{Preliminaries}
First, the following local existence theory, where the initial density is strictly away from vacuum, can be shown by similar arguments as in  \cite{Ma2013}:
\begin{lemma}\label{local}
    Assume that the initial data $(\rho_0, u_0, d_0)$ satisfies the regularity condition \eqref{ia}.
  %   and the compatibility condition
  %   \begin{equation}\label{cc}
		% -\mathrm{div}\left(\mu \rho_0^\alpha \bigl(\nabla u_0+\nabla^\bot u_0\bigr)\right)+\nabla P_0=\rho_0^\frac{1}{2}g,
  %   \end{equation}
  %   for some $(P_0,g)\in H^1\times L^2$. 
    Then there exists a small time $T$ and a unique strong solution $(\rho, u, d)$ to the initial boundary value problem \eqref{ins}-\eqref{bc} such that
\begin{equation}\label{l-r}
    \left\{ \begin{array}{l}
    \rho\in C([0,T];D^{1,2}\cap D^{1,q}),\
    \nabla u\in C([0,T];H^1)\cap L^2([0,T];W^{1,q}),\\
    \rho_t \in C([0,T];L^q),\
    \sqrt{\rho}u_t\in L^\infty([0,T];L^2),\ 
    u_t\in L^2([0,T];H^1_0),\\
    d-e\in C([0,T;H^3])\cap L^2(0,T;H^4), d_t\in C([0,T;H^1])\cap L^2(0,T;H^2).
    \end{array} \right.
\end{equation}
% Furthermore, if $T^\ast$is the maximal existence time of the local strong solution $(\rho, u)$, then either
% $T^\ast=\infty$ or
% \begin{equation}\label{blow-up}
%     \sup_{0\leq t\leq T^\ast}\left(\|\nabla\rho\|_{L^q}+\|\nabla u\|_{L^2}\right)=\infty.
% \end{equation}
\end{lemma}

Next, the following well-known Gagliardo-Nirenberg inequality  will be used frequently later (see \cite{ladyzhenskaia1968linear}).
\begin{lemma}[Gagliardo-Nirenberg Inequality]\label{G-N}
Let 1 $\leq q \le  +\infty$ be a positive extended real quantity. Let $j$ and $m$  be non-negative integers such that $ j < m$. Furthermore, let $1 \leq r \leq \infty$ be a positive extended real quantity,  $p \geq 1$  be real and  $\theta \in [0,1]$  such that the relations
\begin{equation}
    \dfrac{1}{p} = \dfrac{j}{n} + \theta \left( \dfrac{1}{r} - \dfrac{m}{n} \right) + \dfrac{1-\theta}{q}, \qquad \dfrac jm \leq \theta \leq 1,
\end{equation}
hold. Then, 
\begin{equation}
    \|\nabla^j u\|_{L^p(\mathbb{R}^n)} \leq C\|\nabla^m u\|_{L^r(\mathbb{R}^n)}^\theta\|u\|_{L^q(\mathbb{R}^n)}^{1-\theta},
\end{equation}
where $u \in L^q(\mathbb{R}^n)$  such that  $\nabla^m u \in L^r(\mathbb{R}^n)$. Moreover, if $q>1$ and $r>3$,
\begin{equation}
\|u\|_{C(\overline{\mathbb{R}^n})}\leq C\|u\|^{q(r-3)/(3r+q(r-3))}_{L^q(\mathbb{R}^n)}\|\nabla u\|^{3r/(3r+q(r-3))}_{L^r(\mathbb{R}^n)},
\end{equation}
where $u \in L^q(\mathbb{R}^n)$  such that  $\nabla u \in L^r(\mathbb{R}^n)$.
The constant $C > 0$  depends on the parameters $j,\,m,\,n,\,q,\,r,\,$ and $\theta$.

\end{lemma}

\section{A priori estimates}

For any fixed time $T>0$, $(\rho,u, d)$ is the unique local strong solution to \eqref{ins}-\eqref{bc} on $\mathbb{R}^3\times (0,T]$ with initial data $(\rho_0,u_0, d_0)$ satisfying \eqref{ia}, which is guaranteed by Lemma 
\ref{local}.

Define
\begin{gather}\label{As1}
    \mathcal{E}_d(T)\triangleq\sup_{t\in[0,T] }\|\nabla 
    d \|_{L^3},\\
    %\label{As2}\mathcal{E}_u(T) \triangleq \sup_{t\in[0,T] } \bar\rho^{\alpha}\|\nabla u \|_{L^2}^2 +\int_0^{T}  \|\sqrt{\rho} u_t \|_{L^2}^2\,dt,\\
    %\label{As3}\mathcal{E}_\rho(T) \triangleq \sup_{t\in[0,T] }\|\nabla \rho \|_{L^q}.\\
    \mathcal{E}_{\rho,1}(T) \triangleq \sup_{t\in[0,T] }\|\rho-\bar\rho \|_{L^\infty},\\
	\mathcal{E}_{\rho,2}(T) \triangleq \sup_{t\in[0,T] }\|\nabla \rho \|_{L^q}^2
	+ \bar \rho^{\gamma-\alpha} \int_0^{T}  \|\nabla \rho \|_{L^q}^2dt,\\
	\mathcal{E}_{\rho,3}(T) \triangleq \sup_{t\in[0,T] }\|\nabla \rho \|_{L^2}^2
	+ \bar \rho^{\gamma-\alpha} \int_0^{T}  \|\nabla \rho \|_{L^2}^2dt,\\
	\mathcal{E}_{u,1}(T) \triangleq \mu \frac{\bar\rho^{\alpha} }{2^{\alpha+1}} \sup_{t\in[0,T] }\|\nabla u \|_{L^2}^2
	+ \frac 12  \int_0^{T}  \|\sqrt{\rho} u_t \|_{L^2}^2dt,\\
	\mathcal{E}_{u,2}(T) \triangleq \sup_{t\in[0,T] }\|\sqrt{\rho} u_t \|_{L^2}^2
	+ \mu \frac{\bar\rho^{\alpha} }{2^{\alpha+1}}  \int_0^{T}  \|\nabla u_t \|_{L^2}^2dt.
\end{gather}

We have the following key proposition.

\begin{proposition}\label{pr}
Under the conditions of Theorem \ref{global},for 
    \begin{equation}\label{b}
        \beta=\max\{3-\gamma,0\}
    \end{equation}
there exist positive constants  $\Lambda_0$ depending on $a,\mu_1, \mu_2,\nu, \lambda,  \alpha,\gamma$,  $\|G(\rho_0)\|_{L^1}, \|\nabla\rho_0\|_{L^q}, \| u_0\|_{H^2},$ $ \|\nabla d_0\|_{H^2}$, 
and $\varepsilon_0$ depending only on the Sobolev constants, 
and $N_i=N_i(a,\mu_1,\mu_2,\nu, \lambda,  \alpha,\gamma,\|\rho_0-\bar\rho\|_{L^{\gamma}}, \|\nabla \rho_0\|_{L^q}, \|\nabla u_0\|_{H^1}, \|\nabla d_0\|_{H^2}), i=1,2,4$ $N_3=N_3(\mu_1, \mu_2, \|\nabla u_0\|_{L^2})$,  such that if $(\rho,u,d)$ is a smooth solution to the problem (\ref{ins})--(\ref{bc}) on $\mathbb{R}^3\times (0,T]$ satisfying
    \begin{equation}\label{a1}
    \begin{aligned}
        \mathcal{E}_d(T)\leq 2\delta,
        ~\mathcal{E}_{\rho,1}(T) \le \frac{\bar \rho}{2}, 
        ~\mathcal{E}_{\rho,2}(T) \le 2N_1 \bar\rho^{\beta},
        ~\mathcal{E}_{\rho,3}(T) \le 2N_2 \bar\rho^{\beta},\\
	~\mathcal{E}_{u,1}(T) \le 2^{\alpha+2}N_3 \bar\rho^{\alpha},
        ~\mathcal{E}_{u,2}(T) \le 3 N_4 \bar\rho^{2\alpha-1}, \qquad\qquad\quad 
    \end{aligned}
    \end{equation}
for some $\delta<1$ sufficiently small (defined in \eqref{delta}), then the following estimates hold:
    \begin{equation}
        \begin{aligned}
        \mathcal{E}_d(T)\leq \delta,
        ~\mathcal{E}_{\rho,1}(T) \le \frac{\bar \rho}{4}, 
        ~\mathcal{E}_{\rho,2}(T) \le N_1 \bar\rho^{\beta},
        ~\mathcal{E}_{\rho,3}(T) \le N_2 \bar\rho^{\beta},\\
	~\mathcal{E}_{u,1}(T) \le 2^{\alpha+1}N_3 \bar\rho^{\alpha},
        ~\mathcal{E}_{u,2}(T) \le 2 N_4 \bar\rho^{2\alpha-1},  \qquad\qquad\quad
    \end{aligned}
    \end{equation}
	provided 
    \begin{equation}
		\bar\rho\ge \Lambda_0,\quad \|\nabla d_0\|_{L^3}\leq  \varepsilon_0.
    \end{equation}     
\end{proposition}
First, the basic energy inequality of the system \eqref{ins} reads
\begin{lemma}
    Under the assumption $ \mathcal{E}_{\rho,1}(T) \le \frac{\bar \rho}{2}$, it holds that
    \begin{align}\label{basic-est-0}
		&\sup_{0\le t\le T} \left(\bar \rho \| u\|_{L^2}^2 + \bar \rho^{\gamma-2}\|\rho- \bar \rho \|_{L^2}^2 + \|\nabla d\|_{L^2}^2\right)   +  \int_{0}^{T} \bar\rho^{\alpha}\|\nabla u\|_{L^2}^2+\|\Delta d+|\nabla d|^2d\|_{L^2}^2 dt \nonumber\\
        &\le C(\bar\rho\|u_0\|_{L^2}^2+ \|G(\rho_0)\|_{L^1}^\gamma +\|\nabla d_0\|_{L^2}^2) \le C \bar \rho,
	\end{align}
 where $C$ depends on $\|u_0\|_{L^2},\|G(\rho_0)\|_{L^1}, \|\nabla d_0\|_{L^2}$.
 
    Furthermore, under the additional assumptions $\mathcal{E}_{u,1}(T) \le 2^{\alpha+2}N_3 \bar\rho^{\alpha}, \mathcal{E}_d(T)\le 2\delta$ with $\delta$ sufficiently small, we have
    %\eqref{a2} with $\delta$ sufficiently small, we have
	% \begin{equation}\label{basic-est}
	% 	\sup_{0\le t\le T} \left(\bar \rho \| u\|_{L^2}^2+\|\nabla d\|_{L^2}^2\right)   +  \int_{0}^{T} \bar\rho^{\alpha}\|\nabla u\|_{L^2}^2+\|\nabla^2 d\|_{L^2}^2 dt \le C(\bar\rho+\delta).
	% \end{equation}
    \begin{equation}\label{basic-est}
		\sup_{0\le t\le T} \|\nabla d\|_{L^2}^2  +  \int_{0}^{T} \|\nabla^2 d\|_{L^2}^2 dt \le C\|\nabla d_0\|_{L^2}^2,
    \end{equation}
    provided $\bar \rho>1$.
\end{lemma}

\begin{proof}
First, it follows from \eqref{a1} that 
    \begin{equation}\label{upbd}
        \frac{\bar \rho }{2} \le \rho \le \frac{3\bar \rho}{2}.
    \end{equation}
Multiplying \eqref{ins}$_2$ by $u$ and integrating the resultant equation, we obtain after integration by parts that 
 \begin{equation}\label{d1}
    \begin{aligned}
        &\frac{d}{dt} \left( \frac{1}{2} \|\sqrt{\rho} u\|_{L^2}^2 + \|G(\rho)\|_{L^1} \right) + \mu_1 \frac{\bar \rho^{\alpha}}{2^\alpha}\|\nabla u\|_{L^2}^2\\
        &\leq  \frac{d}{dt} \left( \frac{1}{2} \|\sqrt{\rho} u\|_{L^2}^2 + \|G(\rho)\|_{L^1} \right)+2\mu_1 \int\rho^\alpha|\mathcal{D}(u)|^2 dx + \mu_2 \int \rho^\alpha (\mathrm{div} u)^2 dx \\
        &=-\nu \int u\cdot\nabla d\cdot\Delta d\, dx,
    \end{aligned}
\end{equation}
where $G(\rho)$ is defined in \eqref{g}. 
 Multiplying \eqref{ins}$_3$ by $- (\Delta d+|\nabla d|^2d)$ and integrating the resultant equation, we obtain from integration by parts and using $|d|=1$ and \eqref{bc}  that 
	\begin{equation}\label{d-1st}
		\begin{aligned}
			&\frac{1}{2} \frac{d}{dt} \|\nabla d\|_{L^2}^2 
   +\lambda \int|\Delta d+|\nabla d|^2d|^2 dx
   =\int u\cdot\nabla d\cdot\Delta d\,dx.
		\end{aligned}
	\end{equation}
	Adding \eqref{d-1st} multiplied by $\nu$ to \eqref{d1}, integrating over $[0, T]$ and using \eqref{upbd}
    and 
    \begin{equation*}
        G(\rho) \sim \bar \rho^{\gamma-2} (\rho-\bar\rho)^2,
    \end{equation*}
    due to \eqref{upbd}, lead to \eqref{basic-est-0}.
 
 Furthermore, since $|d|=1$ implies $\Delta d\cdot d=-|\nabla d|^2$, one has
 \begin{equation*}
    \int|\Delta d+|\nabla d|^2d|^2\,dx=\int (|\Delta d|^2-|\nabla d|^4)\,dx. 
 \end{equation*}
 % Thus, we obtain from using Lemma \ref{G-N} with $\frac{\partial d}{\partial n}=0$ that
 % \begin{align*}
 %     \|\Delta d\|_{L^2}^2\leq\|\Delta d+|\nabla d|^2d\|_{L^2}^2+\|\nabla d\|_{L^4}^4\leq \|\Delta d+|\nabla d|^2d\|_{L^2}^2+C\|\nabla d\|_{L^3}^2\|\nabla^2 d\|_{L^2}^2
 % \end{align*}
By integration by parts and Sobolev inequality Lemma \ref{G-N}, we have 
\begin{equation}
\begin{aligned}\label{d-ell}
    \|\nabla^2 d\|_{L^2}^2 &=  \|\Delta d\|_{L^2}^2 =  \|\Delta d+|\nabla d|^2d\|_{L^2}^2 +  \|\nabla d\|_{L^4}^4\\
    &\le \|\Delta d+|\nabla d|^2d\|_{L^2}^2 +c_1 \|\nabla d\|_{L^3}^2\|\nabla^2 d\|_{L^2}^2\\
    &\leq  \|\Delta d+|\nabla d|^2d\|_{L^2}^2 + 4\delta^2 c_1 \|\nabla^2 d\|_{L^2}^2\\
    &\leq 2 \|\Delta d+|\nabla d|^2d\|_{L^2}^2,
\end{aligned}
\end{equation}
after choosing 
$$\delta<\frac{1}{2\sqrt{2 c_1}},$$
with $c_1$ depending only on Sobolev constants.
Combining the above result with \eqref{d-1st} gives 
\begin{equation}
    \begin{aligned}
      \frac{d}{dt}\|\nabla d\|_{L^2}^2+\lambda\|\nabla^2 d\|_{L^2}^2
      &\leq C\int|\nabla u||\nabla d|^2 dx
      \leq C\|\nabla u\|_{L^2}\|\nabla d\|_{L^3}\|\nabla d\|_{L^6}\\
      &\leq C\|\nabla u\|_{L^2}\|\nabla d\|_{L^2}^{\frac{1}{2}}\|\nabla^2 d\|_{L^2}^\frac{3}{2}\\
      &\leq \frac{\lambda}{2}\|\nabla^2 d\|_{L^2}^2+C\|\nabla u\|_{L^2}^4\|\nabla d\|_{L^2}^2,
    \end{aligned}
\end{equation}
where we have used 
\begin{equation*}
    \|\nabla d\|_{L^3}\leq C\|\nabla d\|_{L^2}^\frac{1}{2}\|\nabla
    ^2d\|_{L^2}^\frac{1}{2},~~\|\nabla d\|_{L^6}\leq C\|\nabla^2d\|_{L^2}.
\end{equation*}
Applying Gronwall's inequality, we obtain
\begin{equation}\label{359}
    \begin{aligned}
    &\sup_{t\in[0,T]}\|\nabla d\|_{L^2}^2 
    +\lambda \int_{0}^{T} \|\nabla^2 d\|_{L^2}^2 dt \leq \|\nabla d_0\|_{L^2}^2\cdot\exp\left\{C\int_0^T\|\nabla u\|_{L^2}^4 dt\right\}\\
    %&\leq \|\nabla d_0\|_{L^2}^2\exp\left\{\Tilde C(\|\nabla u_0\|_{L^2}, \|\nabla^2 d_0\|_{L^2})\int_0^T\|\nabla u\|_{L^2}^2 dt\right\}\\
    %&\leq C\|\nabla d_0\|_{L^2}^2 \exp\left\{C(\bar\rho^{1-\alpha}\|\nabla u_0\|_{L^2}^2+\bar\rho^{1-\alpha}\|\nabla^2 d_0\|_{L^2}^2)\right\}\\
    &\leq  \|\nabla d_0\|_{L^2}^2
    \exp\left\{C\|\nabla u_0\|_{L^2}^2\bar\rho^{1-\alpha}\right\}\\
    &\leq C\|\nabla d_0\|_{L^2}^2,
    %\leq 2 \|\nabla d_0\|_{L^2}^2, 
    \end{aligned}
\end{equation}
if $\bar \rho >1$.
%if $\bar \rho$ is big enough.
\end{proof}

Next, we can close the a prior assumption $\mathcal{E}_d(T)$.
\begin{lemma}
Under the assumptions that $\mathcal{E}_{u,1}(T) \le 2^{\alpha+2}N_3 \bar\rho^{\alpha}, \mathcal{E}_d(T)\le 2\delta$ with $\delta$ sufficiently small, there exist positive constants $\varepsilon_0$ and $\Lambda_1$ such that 
\begin{equation}
\mathcal{E}_d(T)\le \delta,
\end{equation}
provided $\|\nabla d_0\|_{L^3}\leq  \varepsilon_0\triangleq \frac{\delta}{2}$ and $\bar \rho> \Lambda_1=\Lambda_1(\mu_1,\mu_2, \| \nabla u_0\|_{L^2})$.   
\end{lemma}
\begin{proof}
    Taking the gradient of $\eqref{ins}_3$, we get that
\begin{equation}\label{dd}
    \nabla d_t- \lambda \nabla\Delta d=-\nabla(u\cdot\nabla d)+ \lambda \nabla(\vert\nabla d\vert^2 d).
\end{equation}
Multiplying \eqref{dd} by $|\nabla d| \nabla d$ and integrating the resultant equation give
\begin{equation}
    \begin{aligned}
      &\frac 13 \frac{d}{dt}\|\nabla d\|_{L^3}^3+\lambda \int |\nabla^2 d|^2 |\nabla d| dx+ \lambda \int |\nabla |\nabla d||^2 |\nabla d| dx\\
      &= \lambda \int |\nabla d|^5 dx -  \int \nabla u^i \partial_i d  : \nabla d |\nabla d| dx + \frac 13\int \dv u |\nabla d|^3 dx\\
      &\leq C \lambda \|\nabla d\|_{L^9}^{3} \|\nabla d\|_{L^3}^{2} + C\|\nabla u\|_{L^2}\|\nabla d\|_{L^3}^{\frac 34}\|\nabla d\|_{L^9}^{\frac94}\\
      & \le C \lambda \|\nabla |\nabla d|^{\frac 32}\|_{L^2}^{2} \|\nabla d\|_{L^3}^{2} + C\|\nabla u\|_{L^2}\|\nabla d\|_{L^3}^{\frac 34}\|\nabla |\nabla d|^{\frac 32}\|_{L^2}^{\frac32}\\
      & \le \frac{2\lambda}{9}\|\nabla |\nabla d|^{\frac 32}\|_{L^2}^{2}+ C \lambda \|\nabla d\|_{L^3}^2 \|\nabla |\nabla d|^{\frac 32}\|_{L^2}^{2} + C \|\nabla d\|_{L^3}^3\|\nabla u\|_{L^2}^4\\
      & \le \frac{2\lambda}{9}\|\nabla |\nabla d|^{\frac 32}\|_{L^2}^{2}+ c_2 \lambda \delta \|\nabla |\nabla d|^{\frac 32}\|_{L^2}^{2} + C\|\nabla d\|_{L^3}^3\|\nabla u\|_{L^2}^4,
    \end{aligned}
\end{equation}
which implies 
\begin{equation}
    \begin{aligned}
        &\frac 13 \frac{d}{dt}\|\nabla d\|_{L^3}^3+\lambda \int |\nabla^2 d|^2 |\nabla d| dx+ \frac{\lambda}{9} \int |\nabla |\nabla d|^{\frac 32}|^2 dx \le  C \|\nabla d\|_{L^3}^3 \|\nabla u\|_{L^2}^4,
    \end{aligned}
\end{equation}
after choosing 
\begin{equation}
    \delta< \frac{1}{9 c_2},
\end{equation}
with $c_2$ depending only on Sobolev constants. 
Then, Gronwall inequality implies that 
\begin{equation}
\begin{aligned}
    \sup_{t\in[0,T]} \|\nabla d\|_{L^3} &\le \|\nabla d_0\|_{L^3}\exp\left\{C\int_0^T\|\nabla u\|_{L^2}^4 dt\right\} \\
    &\le \epsilon_0 \exp\left\{C\|\nabla u_0\|_{L^2}^2\bar\rho^{1-\alpha}\right\} \\
    &\le \epsilon_0 \exp\left\{C_1\bar\rho^{1-\alpha}\right\} \le \delta,
\end{aligned}
\end{equation}
after choosing 
\begin{equation}
    \epsilon_0 \triangleq \frac{\delta}{2},
\end{equation}
and 
\begin{equation}
     \bar \rho > \Lambda_1 \triangleq \max\left\{1,\left(\frac{C_1}{\ln2}\right)^{\frac{1}{\alpha-1}}\right\}.
\end{equation}
\end{proof}

\begin{lemma}\label{3d}
Under the assumption $\mathcal{E}_{u,1}(T) \le 2^{\alpha+2}N_3 \bar\rho^{\alpha}, \mathcal{E}_d(T)\le 2\delta$ with $\delta$ sufficiently small, 
we have
\begin{align}\label{d2d}
  \sup_{0\le t\le T}\|\nabla^2 d\|_{L^2}^2 
    +\int_{0}^{T}\|\nabla^3 d\|_{L^2}^2 dt\leq C\|\nabla^2d_0\|_{L^2}^2.
\end{align}
% and
% \begin{align}\label{tdd}
%  \sup_{0\le t\le T}t\|\nabla^2 d\|_{L^2}^2 
%     +\int_{0}^{T}t\|\nabla^3 d\|_{L^2}^2 dt\leq C\|\nabla d_0\|_{L^2}^2.   
% \end{align}
\end{lemma}

\begin{proof}

Multiplying \eqref{dd} by $-\nabla\Delta d $, and using integration by parts, we have 
\begin{equation}\label{d-2nd}
    \begin{aligned}
    &\frac{1}{2}\frac{d}{dt}\|\nabla^2 d\|_{L^2}^2+\lambda \|\nabla\Delta d\|_{L^2}^2=\int\nabla(u\cdot\nabla d)\cdot\nabla \Delta d\,dx-\lambda \int\nabla(|\nabla d|^2d)\nabla\Delta d\,dx\\
    &\leq C (\|\nabla u\|_{L^2}\|\nabla d\|_{L^\infty}+\|u\|_{L^6}\|\nabla^2d\|_{L^3})\|\nabla\Delta d\|_{L^2}\\
    &\quad + C\lambda ( \|\nabla d\|_{L^3}\|\nabla^2d\|_{L^6}+ \|\nabla d\|_{L^\infty}^{\frac{3}{2}}\|\nabla d\|_{L^3}^{\frac{3}{2}})\|\nabla\Delta d\|_{L^2}\\
    &\leq C \|\nabla u\|_{L^2} \|\nabla^2 d\|_{L^2}^{\frac{1}{2}} \|\nabla^3 d\|_{L^2}^{\frac{3}{2}} + C\lambda (\|\nabla d\|_{L^3}+\|\nabla d\|_{L^3}^2)\|\nabla^3d\|_{L^2}^2\\
    &\leq C(\eta+\lambda \|\nabla d\|_{L^3}+\lambda \|\nabla d\|_{L^3}^2)\|\nabla^3d\|_{L^2}^2+C_\eta\|\nabla^2 d\|_{L^2}^2 \|\nabla u\|_{L^2}^4,
\end{aligned}
\end{equation}
where we have used 
$$\|\nabla d\|_{L^\infty} + \|\nabla^2 d\|_{L^3} \le C\|\nabla^3 d\|_{L^2}^{\frac{1}{2}}\|\nabla^2 d\|_{L^2}^\frac{1}{2},$$
\begin{equation}\label{GN-d-2}
   \|\nabla^2 d\|_{L^3}\leq C\|\nabla^3 d\|_{L^2}^{\frac{2}{3}}\|\nabla d\|_{L^3}^\frac{1}{3},~~\|\nabla d\|_{L^\infty}\leq C\|\nabla d\|_{L^3}^\frac{1}{3}\|\nabla^3 d\|_{L^2}^\frac{2}{3}, 
\end{equation}
and
\begin{equation*}
-\int\partial_t\nabla d\cdot\nabla\Delta ddx=\frac{1}{2}\frac{d}{dt}\|\nabla^2 d\|_{L^2}^2.
\end{equation*}
Then choosing $\eta$ small enough, we have
\begin{equation}\label{d-2nd1}
    \begin{aligned}
    \frac{d}{dt}\|\nabla^2 d\|_{L^2}^2
    +\lambda\|\nabla^3 d\|_{L^2}^2
    \leq& \lambda c_3 \delta \|\nabla^3 d\|_{L^2}^2+C\|\nabla^2 d\|_{L^2}^2 \|\nabla u\|_{L^2}^4.
\end{aligned}
\end{equation} 
Then after choosing 
$$\delta<\frac{1}{2c_3},$$
with $c_3$ depending only on Sobolev constants. 
Gronwall's inequality implies 
\begin{equation}\label{d-2nd2}
    \begin{aligned}
    &\sup_{0\le t\le T}\|\nabla^2 d\|_{L^2}^2 
    +\frac{\lambda}{2}\int_{0}^{T}\|\nabla^3 d\|_{L^2}^2 dt
    \leq  \|\nabla^2 d_0\|_{L^2}^2 \cdot \exp{\left\{C\int_{0}^{T}\|\nabla u\|_{L^2}^4dt\right\}} \\
    \le&  \|\nabla^2 d_0\|_{L^2}^2 \cdot \exp\{C \|\nabla u_0\|_{L^2} \bar\rho^{1-\alpha}\} \le C \|\nabla^2 d_0\|_{L^2}^2 .
\end{aligned}
\end{equation}
\end{proof}

% \begin{equation}
% 	\begin{aligned}
% 		&\|\nabla^2 u\|_{L^2}+\bigg\|\nabla\biggl(\frac{P}{\rho^\alpha}\biggr)\bigg\|_{L^2} \\	
% 		\le& C \bar\rho^{-\alpha+\frac12}\|\sqrt{\rho}u_t\|_{L^2}+ \bar\rho^{-\alpha+1}  \|u\cdot \nabla u\|_{L^2} + \bar\rho^{-1} \|\nabla\rho\cdot\nabla u\|_{L^2}
%         +\bar\rho^{-1} \|\nabla\rho\cdot\frac{P}{\rho^\alpha}\|_{L^2}  \\
% 		\le & C \bar\rho^{-\alpha+\frac12}\|\sqrt{\rho}u_t\|_{L^2}+ \bar\rho^{-\alpha+1} \|\nabla u\|_{L^2}^{\frac{3}{2}} \|\nabla u\|_{L^6}^{\frac{1}{2}} + \bar\rho^{-1} \|\nabla \rho \|_{L^q} \|\nabla u\|_{L^{\frac{2q}{q-2}}}\\
% 		\le & \frac{1}{2} \|\nabla u\|_{L^6} + C \bar\rho^{-\alpha+\frac12}\|\sqrt{\rho}u_t\|_{L^2}+ \bar\rho^{-2\alpha+2} \|\nabla u\|_{L^2}^{3}  + \bar\rho^{-\frac{q}{q-3}} \|\nabla \rho \|_{L^q}^{\frac{q}{q-3}} \|\nabla u\|_{L^{2}},
% 	\end{aligned}
% \end{equation}
% and 
% \begin{equation}
% 	\begin{aligned}
% 		&\|\nabla^2 u\|_{L^q} \le C \|H\|_{L^q}\\
% 		\le& C \bar\rho^{-\alpha}\|-\rho(u_t+u\cdot \nabla u)+2\mu \nabla \rho^\alpha \cdot \mathcal{D}u\|_{L^q} \\
% 		\le& C \bar\rho^{-\alpha}\|{\rho}u_t\|_{L^q}+ \bar\rho^{-\alpha+1}  \|u\cdot \nabla u\|_{L^q} + \bar\rho^{-1} \|\nabla \rho \mathcal{D}u\|_{L^q}\\
% 		\le &C \left(\bar\rho^{-\alpha+\frac{5q-6}{4q}}\|\sqrt{\rho}u_t\|_{L^2}^{\frac{6-q}{2q}} \|\nabla u_t\|_{L^2}^{\frac{3(q-3)}{2q}} + \bar\rho^{-\alpha+1}  \| u\|_{L^q} \| \nabla u\|_{L^\infty} + \bar\rho^{-1} \|\nabla \rho\|_{L^q} \| \nabla u\|_{L^\infty}\right),
% 	\end{aligned}
% \end{equation}

Next, we deal with the following high-order estimate of the velocities which will be used frequently.

\begin{lemma}
	Under the assumption $\mathcal{E}_{\rho,2}(T) \le 2N_1 \bar\rho^{\beta}, \mathcal{E}_{\rho,3}(T) \le 2N_2 \bar\rho^{\beta}$, it holds that 
	\begin{equation}\label{H2}
	    \| \nabla  u\|_{L^6} \leq C \left(\bar\rho^{\frac{1}{2}-\alpha}\|\sqrt{\rho} u_t\|_{L^2} +\|\nabla u\|_{L^2} 
        +\bar \rho^{-\alpha+\gamma-1}\|\nabla \rho\|_{L^2}
        %+\bar \rho^{-\alpha+1 + \frac{5}{6}(\gamma-1) } \|\nabla \rho\|_{L^q}^{\frac{5}{7}} 
        +\bar\rho^{-\alpha}\|\nabla d\|_{L^3}\|\nabla^3d\|_{L^2}\right),
	\end{equation}
		and
        \begin{equation}\label{W2p}
        \begin{aligned}
         \|\nabla u\|_{L^\infty} \le&  C \left(\bar\rho^{-\alpha}\|\rho u_t\|_{L^q}  + \bar \rho^{\gamma-\alpha-1}\|\nabla \rho \|_{L^q} + \bar\rho^{-\alpha} \|\nabla^2 d\|_{L^2}^{\frac{3}{q}}\|\nabla^3 d\|_{L^2}^{2-\frac{3}{q}}+\|\nabla u\|_{L^2} \right).
        \end{aligned}
        \end{equation}	
\end{lemma}
\begin{proof}
    Rewrite \eqref{ins}$_2$ as 
\begin{equation}\label{elliptic}
    \mu_1 \Delta u +(\mu_1+\mu_2) \nabla \dv u =  \nabla \mathcal{P}+H,
\end{equation}
with
\begin{equation*}
    \nabla \mathcal{P}=
        \begin{cases}
            \frac{a\gamma}{\gamma-\alpha} \nabla \rho^{\gamma-\alpha}, \quad \gamma \neq \alpha,\\
            a\gamma \nabla \ln \rho, \quad \gamma=\alpha,
        \end{cases}
\end{equation*}
and
$$H \triangleq \rho^{1-\alpha}\dot u
+\alpha\rho^{-1}\nabla\rho\cdot(2\mu \mathcal{D}u +\lambda \dv u) - \nu \rho^{-\alpha} \nabla d \cdot \Delta d .$$
Define the efficient viscous flux $F$ and vorticity $w$ respect to \eqref{elliptic} as
\begin{equation}\label{flux}
    F=(2\mu_1+\mu_2) \dv u +  \mathcal{P}(\rho)-\mathcal{P}(\bar \rho),\quad w=\nabla \times u,
\end{equation}
then \eqref{elliptic} implies
\begin{gather}
    \label{f1}-\Delta F=\dv H,\\
    -\mu \Delta w =\nabla \times H.
\end{gather}

According to the standard $L^2$-estimate of the elliptic system, one has after using \eqref{upbd}, \eqref{basic-est-0}, and Sobolev's inequality in Lemma \ref{G-N},
\begin{equation}\label{e1}
    \begin{aligned}
        &\|\nabla u\|_{L^6} \le C \|\nabla^2 u\|_{L^2}\\
        %\le& C (\|F\|_{L^6} + \|w\|_{L^6}+ \|\mathcal{P}(\rho)-\mathcal{P}(\bar \rho)\|_{L^6}) \\
        \le & C( \|\nabla F\|_{L^2} + \|\nabla w\|_{L^2}+ \bar \rho^{-\alpha}\|\nabla P(\rho)\|_{L^2} )\\
        \le & C \left(\bar\rho^{-\alpha+\frac12}\|\sqrt{\rho}u_t\|_{L^2}+ \bar\rho^{-\alpha+1} \|\nabla u\|_{L^2}^{\frac{3}{2}} \|\nabla u\|_{L^6}^{\frac{1}{2}} + \bar\rho^{-1} \|\nabla \rho \|_{L^q} \|\nabla u\|_{L^{2}}^{\frac{q-3}{q}} \|\nabla u\|_{L^6}^{\frac{3}{q}} \right) \\
        &+ C \left(\bar \rho^{-\alpha+\gamma-1}\|\nabla \rho\|_{L^2} +\bar\rho^{-\alpha}\|\nabla d\|_{L^3}\|\nabla^3d\|_{L^2} \right)\\
        \le & \frac{1}{2} \|\nabla u\|_{L^6} + C \left( \bar\rho^{-\alpha+\frac12}\|\sqrt{\rho}u_t\|_{L^2}+ \bar\rho^{-2\alpha+2} \|\nabla u\|_{L^2}^{3}  + \bar\rho^{-\frac{q}{q-3}} \|\nabla \rho \|_{L^q}^{\frac{q}{q-3}} \|\nabla u\|_{L^{2}} \right)\\
        &+ C \left(\bar \rho^{-\alpha+\gamma-1}\|\nabla \rho\|_{L^2} + \bar\rho^{-\alpha}\|\nabla d\|_{L^3}\|\nabla^3d\|_{L^2} \right),
    \end{aligned}
\end{equation}
where we have used  
\begin{equation}\label{e2}
    \begin{aligned}
        &\|\nabla F\|_{L^2} + \|\nabla w\|_{L^2} \le C \|H\|_{L^2}\\
        \le& C \bar\rho^{-\alpha}\|-\rho(u_t+u\cdot \nabla u)+2\mu \nabla \rho^\alpha \cdot \mathcal{D}u +\lambda \nabla \rho^{\alpha} \dv u - \nu  \nabla d \cdot \Delta d\|_{L^2}\\
        \le & C \left(\bar\rho^{-\alpha+\frac12}\|\sqrt{\rho}u_t\|_{L^2}+ \bar\rho^{-\alpha+1} \|\nabla u\|_{L^2}^{\frac{3}{2}} \|\nabla u\|_{L^6}^{\frac{1}{2}}\right)\\
        & + C \left( \bar\rho^{-1} \|\nabla \rho \|_{L^q} \|\nabla u\|_{L^{\frac{2q}{q-2}}}  + \bar\rho^{-\alpha}\|\nabla d\|_{L^3}\|\nabla^3d\|_{L^2} \right)\\
        \le & C \left( \bar\rho^{-\alpha+\frac12}\|\sqrt{\rho}u_t\|_{L^2}+ \bar\rho^{-\alpha+1} \|\nabla u\|_{L^2}^{\frac{3}{2}} \|\nabla u\|_{L^6}^{\frac{1}{2}} \right)\\
        & + \left( \bar\rho^{-1} \|\nabla \rho \|_{L^q} \|\nabla u\|_{L^{2}}^{\frac{q-3}{q}} \|\nabla u\|_{L^6}^{\frac{3}{q}} +\bar\rho^{-\alpha}\|\nabla d\|_{L^3}\|\nabla^3d\|_{L^2}\right).
    \end{aligned}
\end{equation}

Note that by Sobolev's inequality,
    \begin{equation}\label{rho1}
        \begin{aligned}
        \|\nabla u\|_{L^\infty} \le & C \|\nabla u\|_{L^2}^{\theta} \|\nabla^2 u\|_{L^q}^{1-\theta}\\
        %\le & C \|\nabla u\|_{L^2}^{\theta} \left( \bar\rho^{-\alpha}\|\rho u_t\|_{L^q} + \bar\rho^{\frac{-\alpha+1}{\theta}}  \| u\|_{L^q}^{\frac{1}{\theta}} \|\nabla u\|_{L^2} + \bar\rho^{-\frac{1}{\theta}} \|\nabla \rho\|_{L^q}^{\frac{1}{\theta}} \|\nabla u\|_{L^2}  + \bar \rho^{\gamma-\alpha-1}\|\nabla \rho \|_{L^q}\right)^{1-\theta}
        \end{aligned}
    \end{equation}
    where 
    \begin{equation}\label{theta}
        \theta=\frac{2(q-3)}{5q-6}.
    \end{equation}
    
    Similarly, the standard $L^p$-estimate of the elliptic system gives that
    \begin{equation}\label{e-p}
        \begin{aligned}
         \|\nabla^2 u\|_{L^q} \le & C( \|\nabla F\|_{L^q} + \|\nabla w\|_{L^q}  + \|\nabla \mathcal{P}\|_{L^q})\\
        %\le & C (\|H\|_{L^q} + \|\nabla \mathcal{P}\|_{L^q})\\
        %\le& C \bar\rho^{-\alpha}\|-\rho(u_t+u\cdot \nabla u)+2\mu \nabla \rho^\alpha \cdot \mathcal{D}u +\lambda \nabla \rho^{\alpha} \dv u\|_{L^q} \\
        \le& C \left(\bar\rho^{-\alpha}\|\rho u_t\|_{L^q} + \bar \rho^{\gamma-\alpha-1}\|\nabla \rho \|_{L^q} + \bar \rho^{-\alpha}\|\nabla d\|_{L^{\frac{6q}{6-q}}}\|\nabla^2d\|_{L^6}\right)\\
        & +C \left( \bar\rho^{-\alpha+1}  \| u\|_{L^q} \| \nabla u\|_{L^\infty} + \bar\rho^{-1} \|\nabla \rho\|_{L^q} \| \nabla u\|_{L^\infty} \right)\\
        \le& C \left(\bar\rho^{-\alpha}\|\rho u_t\|_{L^q} + \bar \rho^{\gamma-\alpha-1}\|\nabla \rho \|_{L^q} + \bar \rho^{-\alpha} \|\nabla d\|_{L^3}^{\frac{2}{q}}\|\nabla^2 d\|_{L^6}^{2-\frac{2}{q}}\right)\\
        & + C \left( \bar\rho^{-\alpha+1}  \| u\|_{L^q} \|\nabla u\|_{L^2}^{\theta} \|\nabla^2 u\|_{L^q}^{1-\theta} + \bar\rho^{-1} \|\nabla \rho\|_{L^q} \|\nabla u\|_{L^2}^{\theta} \|\nabla^2 u\|_{L^q}^{1-\theta} \right)\\
        \le & \frac{1}{2} \|\nabla^2 u\|_{L^q} + C \left(\bar\rho^{-\alpha}\|\rho u_t\|_{L^q}  + \bar \rho^{\gamma-\alpha-1}\|\nabla \rho \|_{L^q} + \bar\rho^{-\alpha} \|\nabla d\|_{L^3}^{\frac{2}{q}}\|\nabla^3 d\|_{L^2}^{2-\frac{2}{q}}\right)\\
         & +  C \left( \bar\rho^{\frac{-\alpha+1}{\theta}}  \| u\|_{L^q}^{\frac{1}{\theta}} \|\nabla u\|_{L^2} + \bar\rho^{-\frac{1}{\theta}} \|\nabla \rho\|_{L^q}^{\frac{1}{\theta}} \|\nabla u\|_{L^2} \right),
        \end{aligned}
    \end{equation}
    where we have used 
    \begin{equation}\label{e3}
    \begin{aligned}
        &\|\nabla F\|_{L^q} + \|\nabla w\|_{L^q} \le C \|H\|_{L^q}\\
        \le& C \bar\rho^{-\alpha}\|-\rho(u_t+u\cdot \nabla u)+2\mu \nabla \rho^\alpha \cdot \mathcal{D}u +\lambda \nabla \rho^{\alpha} \dv u - \nu  \nabla d \cdot \Delta d\|_{L^q}\\
        \le& C \left(\bar\rho^{-\alpha}\| \rho u_t\|_{L^q}  + \bar\rho^{-\alpha+1}  \| u\|_{L^q} \| \nabla u\|_{L^\infty} + \bar\rho^{-1} \|\nabla \rho\|_{L^q} \| \nabla u\|_{L^\infty} + \bar \rho^{-\alpha}\|\nabla d\|_{L^{\frac{6q}{6-q}}}\|\nabla^2d\|_{L^6}\right).
    \end{aligned}
    \end{equation}

    Finally, it follows from \eqref{e1} and \eqref{e-p} that
    \begin{equation}\label{H2}
    \begin{aligned}
        &\| \nabla  u\|_{L^6} \\
        \leq & C \bigg(\bar\rho^{\frac{1}{2}-\alpha}\|\sqrt{\rho} u_t\|_{L^2} +\bar\rho^{2-2\alpha}\|\nabla u\|_{L^2}^3 + \bar\rho^{-\frac{q}{q-3}} \|\nabla \rho \|_{L^q}^{\frac{q}{q-3}} \|\nabla u\|_{L^{2}} \\
        &+ \bar \rho^{-\alpha+\gamma-1}\|\nabla \rho\|_{L^2} +\bar\rho^{-\alpha}\|\nabla d\|_{L^3}\|\nabla^3d\|_{L^2} \bigg)\\
        %\le C (\bar\rho^{\frac{1}{2}-\alpha}\|\sqrt{\rho} u_t\|_{L^2} +\|\nabla u\|_{L^2} +\bar \rho^{-\alpha + \gamma } \|\nabla \rho\|_{L^q}^{\frac{5q}{2(4q-3)}} +\bar\rho^{-\alpha}\|\nabla d\|_{L^3}\|\nabla^3d\|_{L^2})\\
        \le & C \left(\bar\rho^{\frac{1}{2}-\alpha}\|\sqrt{\rho} u_t\|_{L^2} +\|\nabla u\|_{L^2} 
        +\bar \rho^{-\alpha+\gamma-1}\|\nabla \rho\|_{L^2} + \bar\rho^{-\alpha}\|\nabla d\|_{L^3}\|\nabla^3d\|_{L^2} \right),
    \end{aligned}
    \end{equation}
    and 
    \begin{equation}\label{W2p}
        \begin{aligned}
         \|\nabla^2 u\|_{L^q} \le&  C \left(\bar\rho^{-\alpha}\|\rho u_t\|_{L^q}  + \bar \rho^{\gamma-\alpha-1}\|\nabla \rho \|_{L^q} + \bar\rho^{-\alpha} \|\nabla^2 d\|_{L^2}^{\frac{3}{q}}\|\nabla^3 d\|_{L^2}^{2-\frac{3}{q}}\right)\\
         & +  C \left( \bar\rho^{\frac{-\alpha+1}{\theta}}  \| u\|_{L^q}^{\frac{1}{\theta}} \|\nabla u\|_{L^2} + \bar\rho^{-\frac{1}{\theta}} \|\nabla \rho\|_{L^q}^{\frac{1}{\theta}} \|\nabla u\|_{L^2} \right)\\
         \le&  C \left(\bar\rho^{-\alpha}\|\rho u_t\|_{L^q}  + \bar \rho^{\gamma-\alpha-1}\|\nabla \rho \|_{L^q} + \bar\rho^{-\alpha} \|\nabla^2 d\|_{L^2}^{\frac{3}{q}}\|\nabla^3 d\|_{L^2}^{2-\frac{3}{q}}+\|\nabla u\|_{L^2} \right),
        \end{aligned}
        \end{equation}
        due to $\alpha>1$ , $\beta<2$ and $q\in(3,6)$.
\end{proof}

Now we are ready to deal with an estimate of $\mathcal{E}_{u,1}(T)$.
\begin{lemma}\label{L_2}
Under the assumption \eqref{a1}, and 
\begin{equation}
    \alpha>\frac{\gamma+1}{2}, \alpha\ge \gamma-1,
\end{equation}
then there exists a positive constant $\Lambda_2$ such that 
\begin{equation}
\mathcal{E}_{u,1}(T)\le 2^{\alpha+1}N_3 \bar\rho^{\alpha},
\end{equation}
% and
% \begin{equation}\label{tdu}
% \sup_{t\in[0,T]}t\bar\rho^{\alpha}\|\nabla u \|_{L^2}^2
% +\int_0^{T}t \|\sqrt{\rho} u_t \|_{L^2}^2 dt \leq C\bar\rho ,
% \end{equation}
% \begin{equation}\label{tdd}
% \sup_{t\in[0,T]}t \|\nabla^2d\|_{L^2}^2
% +\int_0^{T}t \|\nabla^3d\|_{L^2}^2 dt \leq C \|\nabla d_0\|_{L^3}^2,
% \end{equation}
with $N_3$ defined in \eqref{m1}, provided $\bar \rho\ge \Lambda_2= \Lambda_2(\mu_1,\mu_2,\nu, \lambda, \alpha,a, \|\nabla \rho_0\|_{L^q}, \| u_0\|_{H^1}, \|\nabla^2 d_0\|_{L^2})$.
\end{lemma}
\begin{proof}
	Multiplying \eqref{ins}$_2$ by $u_t$, and integrating by parts, we have that 
	\begin{equation}\label{51}
		\begin{aligned}
			&\frac{d}{dt}\int (\mu_1 \rho^\alpha |\mathcal{D}(u)|^2 + \mu_2 \rho^\alpha |\dv u|^2 )dx -\frac{d}{dt}\int  (P(\rho)-P(\bar\rho)) \dv u  dx +  \int\rho |u_t|^2dx \\
			&= -\int  \rho u \cdot \nabla u \cdot u_t dx+ \int  \mu_1(\rho^\alpha)_t |\mathcal{D}(u)|^2 dx + \int \mu_2 (\rho^\alpha)_t|\dv u|^2 dx - \int P_t \dv u dx\\
            &\quad -\nu \int u_t \cdot  \nabla d \cdot \Delta d  dx.
		\end{aligned}
	\end{equation}
	It follows from H\"older and Sobolev inequalities that 
	\begin{equation}\label{52}
        \begin{aligned}
		&\int  \rho u \cdot \nabla u \cdot u_t dx\\
        &\leq C\bar\rho^\frac{1}{2}\|\sqrt{\rho} u_t\|_{L^2}\|u\|_{L^6}\|\nabla u\|_{L^3} \\
        &\leq \frac{1}{16}\|\sqrt{\rho} u_t\|_{L^2}^2+C \bar\rho \|\nabla u\|_{L^2}^3\|\nabla u\|_{L^6}.\\        
        %&\leq \frac{1}{16}\|\sqrt{\rho} u_t\|_{L^2}^2+ C\bar \rho \|\nabla u\|_{L^2}^3 (\bar\rho^{\frac{1}{2}-\alpha}\|\sqrt{\rho} u_t\|_{L^2} +\|\nabla u\|_{L^2} + \bar \rho^{-\alpha + \gamma } \|\nabla \rho\|_{L^q}^{\frac{5q}{2(4q-3)}} +\bar\rho^{-\alpha}\|\nabla d\|_{L^3}\|\nabla^3d\|_{L^2})\\
        %&\leq \frac{1}{8}\|\sqrt{\rho} u_t\|_{L^2}^2 +C(\bar\rho^{3-2\alpha}+\bar\rho^{2-2\alpha})\|\nabla u\|_{L^2}^6 + C\bar \rho \|\nabla u\|_{L^2}^4 + C \bar \rho^{\gamma+1-\alpha}\|\nabla u\|_{L^2}^3 +C\|\nabla d\|_{L^3}^2\|\nabla^3d\|_{L^2}^2\\
        %&\leq \frac{1}{8}\|\sqrt{\rho} u_t\|_{L^2}^2 +C\bar\rho^{\alpha}\|\nabla u\|_{L^2}^2 +C\|\nabla d\|_{L^3}^2\|\nabla^3d\|_{L^2}^2.
        \end{aligned}
	\end{equation}
Using 
% \begin{equation}\label{rhoat}
%     \partial_t(\rho^\alpha)+u\cdot\nabla\rho^\alpha + \alpha \rho^\alpha \dv u=0,
% \end{equation}
\eqref{ins}$_1$, which together with Sobolev inequalities yields
\begin{equation}\label{53}
\begin{aligned}
&\int\mu_1(\rho^\alpha)_t |\mathcal{D}(u)|^2 dx + \int \mu_2 (\rho^\alpha)_t|\dv u|^2 dx\\
&\le C \bar \rho^{\alpha-1} \int  \left|\rho\dv u+\nabla \rho \cdot u\right||\nabla u|^2 dx\\
&\le C\bar\rho^\alpha  \|\nabla u\|_{L^3}^3 + C \bar \rho^{\alpha-1} \|\nabla \rho \|_{L^q}\|u\|_{L^6} \|\nabla u\|_{L^{\frac{12q}{5q-6}}}^2\\
&\le C \bar \rho^\alpha \|\nabla u\|_{L^2}^{\frac{3}{2}} \|\nabla u\|_{L^6}^{\frac{3}{2}} + C\bar{\rho}^{\alpha-1} \|\nabla \rho \|_{L^q}\|\nabla u\|_{L^2}^{\frac{5q-6}{2q}} \|\nabla u\|_{L^6}^{\frac{q+6}{2q}}\\
&\le C \bar \rho^\alpha \|\nabla u\|_{L^2}^{\frac{3}{2}} \|\nabla u\|_{L^6}^{\frac{3}{2}} + C\bar{\rho}^{\alpha} \|\nabla u\|_{L^2}^{2} \|\nabla u\|_{L^6},
%&\le C\bar \rho^\alpha \|\nabla u\|_{L^2}^{\frac{3}{2}} (\bar\rho^{\frac{1}{2}-\alpha} \|\sqrt{\rho} u_t\|_{L^2} +\|\nabla u\|_{L^2} + \bar \rho^{-\alpha + \gamma } \|\nabla \rho\|_{L^q}^{\frac{5q}{2(4q-3)}} +\bar\rho^{-\alpha}\|\nabla d\|_{L^3}\|\nabla^3d\|_{L^2} )^{\frac{3}{2}}\\
%&\quad + C\mathcal{E}_\rho(0)\bar{\rho}^{\alpha-1} \|\nabla u\|_{L^2}^{\frac{5q-6}{2q}} (\bar\rho^{\frac{1}{2}-\alpha} \|\sqrt{\rho} u_t\|_{L^2}
%+\|\nabla u\|_{L^2} + \bar \rho^{-\alpha + \gamma } \|\nabla \rho\|_{L^q}^{\frac{5q}{2(4q-3)}} +\bar\rho^{-\alpha}\|\nabla d\|_{L^3}\|\nabla^3d\|_{L^2} )^{\frac{q+6}{2q}}\\
%&\leq \frac{1}{16} \|\sqrt{\rho} u_t\|_{L^2}^2 + C\bar \rho^{3-2\alpha} \|\nabla u\|_{L^2}^6 +  C\bar \rho^{\alpha} \|\nabla u\|_{L^2}^3  + C \bar \rho^{(\gamma-1)\frac{q+6}{2q}-(\alpha-1)\frac{q-6}{2q} } \|\nabla \rho\|_{L^q}^{\frac{5(q+6)}{4(4q-3)}} \|\nabla u\|_{L^2}^{\frac{3}{2}}\\
%&+C\bar\rho^{-\frac{q+6}{3q-6}+\frac{12-2q}{3q-6}(1-\alpha)} \|\nabla u\|_{L^2}^\frac{2(5q-6)}{3(q-2)} + C \bar \rho^{\frac{3\gamma-\alpha}{2} } \|\nabla \rho\|_{L^q}^{\frac{15q}{4(4q-3)}} \|\nabla u\|_{L^2}^{\frac{5q-6}{2q}}
%+C\bar \rho^{-1}\|\nabla d\|_{L^3}^2\|\nabla^3d\|_{L^2}^2\\
%&\le \frac{1}{16} \|\sqrt{\rho} u_t\|_{L^2}^2 + C\bar \rho^{\alpha} \|\nabla u\|_{L^2}^2  + C \bar \rho^{(\gamma-1)\frac{q+6}{2q}-(\alpha-1)\frac{q-6}{2q} } \|\nabla \rho\|_{L^q}^{\frac{5(q+6)}{4(4q-3)}} \|\nabla u\|_{L^2}^{\frac{3}{2}}\\
%&+ C \bar \rho^{\frac{3\gamma-\alpha}{2} } \|\nabla \rho\|_{L^q}^{\frac{15q}{4(4q-3)}} \|\nabla u\|_{L^2}^{\frac{5q-6}{2q}} +C\|\nabla d\|_{L^3}^2\|\nabla^3d\|_{L^2}^2,
\end{aligned}
\end{equation}
due to $\beta<2, q\in(3,6)$.
Similarly,
    \begin{equation}
        \begin{aligned}
        \int P_t \dv u dx&\le C \bar \rho^{\gamma-1} \int  \left|\rho \dv u+ \nabla \rho \cdot u\right||\dv u| dx\\
        &\le C \bar \rho^\gamma \|\nabla u\|_{L^2}^2 + C \bar \rho^{\gamma-1} \|\nabla \rho \|_{L^q}\|u\|_{L^\frac{2q}{q-2}} \|\nabla u\|_{L^2}\\
        &\le C\bar \rho^\gamma \|\nabla u\|_{L^2}^{2}+ C \bar \rho^{\gamma-1} \|\nabla \rho \|_{L^q}\|u\|_{L^2}^{\frac{q-3}{q}} \|\nabla u\|_{L^2}^{\frac{3+q}{q}}\\
        &\le C\bar \rho^\gamma \|\nabla u\|_{L^2}^{2}+ C \bar \rho^{\gamma-1} \|\nabla \rho \|_{L^q} \|\nabla u\|_{L^2}\\
        &\le C\bar \rho^\gamma \|\nabla u\|_{L^2}^{2} + C \bar \rho^{\gamma-2} \|\nabla \rho \|_{L^q}^2,
        \end{aligned}
    \end{equation}
    where we have used $q\in(3,6)$ and \eqref{a1}.

Using integration by parts and $\eqref{ins}_3$ yield that
\begin{equation}
\begin{aligned}
    &-\int  u_t \cdot \nabla d \cdot \Delta d dx=\int \nabla d\odot\nabla d:\nabla u_t dx - \frac 12 \int \dv u_t |\nabla d|^2dx
\\
    &=\frac{d}{dt} \left(\int\nabla d\odot\nabla d:\nabla u dx - \frac 12 \int \dv u |\nabla d|^2dx \right)\\
    &\quad -2\int\nabla d_t \odot\nabla d:\mathcal{D}(u) dx + \int  \nabla d_t : \nabla d \dv u dx
\\
    &\leq \frac{d}{dt} \left(\int\nabla d\odot\nabla d:\nabla u dx - \frac 12 \int \dv u |\nabla d|^2dx \right) \\
    &\quad+C \int|\nabla u|^2|\nabla d|^2 dx+ C\int|u||\nabla^2 d||\nabla d||\nabla u| dx
\\
    &\quad +C\int|\nabla\Delta d||\nabla d||\nabla u| dx+C\int|\nabla d|^2|\nabla^2d||\nabla u| dx+C\int|\nabla d|^4|\nabla u| dx\\
    &=\frac{d}{dt} \left(\int\nabla d\odot\nabla d:\nabla u dx - \frac 12 \int \dv u |\nabla d|^2dx \right) +\sum_{i=1}^5I_i.
\end{aligned}    
\end{equation}
It follows from H\"older's inequality, Gagliardo-Nirenberg inequality, and \eqref{H2} that
\begin{equation}
\begin{aligned}\label{I-1}
    I_1+I_2&\leq C\|\nabla d\|_{L^\infty}^2\|\nabla u\|_{L^2}^2+ C\|u\|_{L^6}\|\nabla^2 d\|_{L^6}\|\nabla d\|_{L^6}\|\nabla u\|_{L^2}\\
        &\leq C\|\nabla u\|_{L^2}^2\|\nabla^3 d\|_{L^2}^{\frac{4}{3}}\|\nabla d\|_{L^3}^\frac{2}{3}\\
        &\leq C\|\nabla d\|_{L^3}\|\nabla^3d\|_{L^2}^2
        +C\|\nabla u\|_{L^2}^6,
    \end{aligned}
\end{equation} 
and 
\begin{equation}\label{I-5}
   \begin{aligned}
        I_3+I_4+I_5&\leq C( \|\nabla\Delta d\|_{L^2}\|\nabla d\|_{L^\infty}+\|\nabla d\|_{L^6}^2\|\nabla^2 d\|_{L^6}+\|\nabla d\|_{L^8}^{4}) \|\nabla u\|_{L^2}\\
        &\leq C(\|\nabla d\|_{L^3}^{\frac 13}\|\nabla^3d\|_{L^2}^{\frac{5}{3}} + \|\nabla d\|_{L^3}^{\frac 43}\|\nabla^3d\|_{L^2}^{\frac{5}{3}} +\|\nabla d\|_{L^3}^{\frac 73}\|\nabla^3d\|_{L^2}^{\frac{5}{3}}) \|\nabla u\|_{L^2}\\
        &\leq C\|\nabla d\|_{L^3}^{\frac 25}\|\nabla^3d\|_{L^2}^2+C \|\nabla u\|_{L^2}^6,
    \end{aligned}
\end{equation}
where we have used \eqref{GN-d-2}.

Then \eqref{a1} together with \eqref{51}-\eqref{I-5} implies that 
\begin{equation}\label{t1}
\begin{aligned}
&\frac{d}{dt} B(t) +  \frac{3}{8}\int\rho |u_t|^2dx\\
&\leq C(\bar \rho+\bar \rho^{\gamma})\|\nabla u\|_{L^2}^2   + C \bar \rho^{\gamma-2} \|\nabla \rho \|_{L^q}^2 +C\|\nabla d\|_{L^3}^{\frac{2}{5}}\|\nabla^3d\|_{L^2}^2\\
%& \quad+ C \bar \rho^{(\gamma-1)\frac{q+6}{2q}-(\alpha-1)\frac{q-6}{2q} } \|\nabla \rho\|_{L^q}^{\frac{5(q+6)}{4(4q-3)}} \|\nabla u\|_{L^2}^{\frac{3}{2}}+ C \bar \rho^{\frac{3\gamma-\alpha}{2} } \|\nabla \rho\|_{L^q}^{\frac{15q}{4(4q-3)}} \|\nabla u\|_{L^2}^{\frac{5q-6}{2q}}
&\quad +C \bar \rho^\alpha \|\nabla u\|_{L^2}^{\frac{3}{2}} \|\nabla u\|_{L^6}^{\frac{3}{2}} + C(\bar \rho + \bar{\rho}^{\alpha}) \|\nabla u\|_{L^2}^{2} \|\nabla u\|_{L^6},
\end{aligned}
\end{equation}
where
\begin{align*}
    B(t)=\int (\mu_1\rho^\alpha |\mathcal{D}(u)|^2 + \mu_2 \rho^\alpha |\dv u|^2 )dx -\int (P(\rho)-P(\bar\rho)) \dv u dx \\
    - \lambda \int(\nabla d\odot\nabla d:\nabla u - \frac 12 \dv u |\nabla d|^2)dx.
\end{align*}
Note that 
\begin{align*}
%    &\frac{\mu_1}{2^{\alpha+2}} \bar \rho^{\alpha} \|\nabla u\|_{L^2}^2 -  C \bar \rho^{{\gamma}-\alpha} \|P(\rho)-P(\bar \rho)\|_{L^1} - C \bar \rho^{-\alpha} \|\nabla d\|_{L^4}^4\\
    B(t) 
    \le &C \bar \rho^{\alpha} \|\nabla u\|_{L^2}^2 + C  \|P(\rho)-P(\bar \rho)\|_{L^2} \|\dv u\|_{L^2}+ C  \|\nabla d\|_{L^4}^2 \|\dv u\|_{L^2}\\
    \le &C \bar \rho^{\alpha} \|\nabla u\|_{L^2}^2 + C \bar \rho^{\gamma-1} \|\rho- \bar \rho\|_{L^2} \|\dv u\|_{L^2} + C  \|\nabla d\|_{L^3} \|\nabla^2 d\|_{L^2} \|\dv u\|_{L^2},
\end{align*}
and 
\begin{equation}\label{k1}
    \begin{aligned}
        &\int_{0}^{T} \|\nabla u\|_{L^2}^{2} \|\nabla u\|_{L^6}  dt\\
        \le & C\int_{0}^{T} \|\nabla u\|_{L^2}^{2} \left(\bar\rho^{\frac{1}{2}-\alpha} \|\sqrt{\rho} u_t\|_{L^2}
        +\|\nabla u\|_{L^2} +\bar \rho^{-\alpha+\gamma-1}\|\nabla \rho\|_{L^2}  \right) dt\\
        & + C \bar\rho^{-\alpha} \int_{0}^{T} \|\nabla u\|_{L^2}^{2} \|\nabla d\|_{L^3}\|\nabla^3d\|_{L^2}  dt\\
        \le & C \int_{0}^{T} ( \|\nabla u\|_{L^2}^2 + \bar \rho^{-\alpha + \gamma -1+\frac{\beta}{2}} \|\nabla u\|_{L^2}^{2} + \bar \rho^{-\alpha} \|\nabla u\|_{L^2} \|\nabla^3d\|_{L^2} ) dt\\
        \le & C \left( \bar \rho^{1-\alpha}+ \bar \rho^{-\alpha + \gamma  + (1-\alpha)} + \bar \rho^{-\alpha+ \frac 12 (1-\alpha)}\right) \\
        \le & C \left(\bar \rho^{1-\alpha} + \bar \rho^{\gamma-2\alpha+1} \right),
    \end{aligned}
\end{equation}
due to $\beta<1$. 
Similarly,
\begin{equation}\label{k2}
    \begin{aligned}
        &\int_{0}^{T} \|\nabla u\|_{L^2}^{\frac{3}{2}} \|\nabla u\|_{L^6}^{\frac{3}{2}}  dt\\
        \le & C\int_{0}^{T} \|\nabla u\|_{L^2}^{\frac{3}{2}} \left(\bar\rho^{\frac{1}{2}-\alpha} \|\sqrt{\rho} u_t\|_{L^2} + \|\nabla u\|_{L^2} + \bar \rho^{-\alpha+\gamma-1}\|\nabla \rho\|_{L^2} \right)^{\frac{3}{2}} dt\\
        & + C\int_{0}^{T} \bar\rho^{-\frac{3\alpha}{2}} \|\nabla u\|_{L^2}^{\frac{3}{2}}  \|\nabla d\|_{L^3}^{\frac{3}{2}} \|\nabla^3d\|_{L^2}^{\frac{3}{2}} dt\\
        \le & C \int_{0}^{T} \left(\bar \rho^{\frac12 (\frac{1}{2}-\alpha)} \|\nabla u\|_{L^2}^{\frac{3}{2}} \|\sqrt{\rho} u_t\|_{L^2}^{\frac 12} + \|\nabla u\|_{L^2}^2 + \bar \rho^{-\frac32 \alpha +\frac{3}{2} \gamma -\frac 32+ \frac{\beta}{2} } \|\nabla u\|_{L^2}^{\frac{3}{2}} \|\nabla \rho\|_{L^q}^{\frac{1}{2}}  \right) dt\\
        & + C \int_{0}^{T}   \bar \rho^{-\frac{3}{2}\alpha} \|\nabla u\|_{L^2}^{\frac 12}\|\nabla^3d\|_{L^2}^{\frac{3}{2}} dt\\
        \le & C \left(\bar \rho^{\frac 12 (\frac{1}{2}-\alpha) + (1-\alpha) \frac{3}{4} + \frac 14 \alpha } + \bar \rho^{1-\alpha}+ \bar \rho^{-\frac32 \alpha +\frac{3}{2} \gamma -\frac 32+ \frac{\beta}{2} + (1-\alpha) \frac{3}{4} + \frac 14 (\alpha-\gamma+\beta)} + \bar \rho^{-\frac 32 \alpha + \frac 14 (1-\alpha)} \right) \\
        \le & C \left(\bar \rho^{1-\alpha} + \bar \rho^{-2 \alpha +\frac{5}{4} \gamma + \frac{3\beta}{4} -   \frac{3}{4} } \right)\\
        \le & \begin{cases}
            C \left(\bar \rho^{1-\alpha} + \bar \rho^{-2 \alpha +\frac{5}{4} \gamma -   \frac{3}{4} } \right),\quad \gamma>3\\
            C \left(\bar \rho^{1-\alpha} + \bar \rho^{-2\alpha+\frac 12 \gamma+\frac 32} \right),\quad 1<\gamma\le 3
        \end{cases}\\
        \le & \begin{cases}
            C \left(\bar \rho^{1-\alpha} + \bar \rho^{-\frac 3 4 \alpha +\frac{1}{2} } \right),\quad \gamma>3\\
            C \left(\bar \rho^{1-\alpha} + \bar \rho^{-2\alpha+\frac 12 \gamma+\frac 32} \right),\quad 1<\gamma\le 3
        \end{cases}\\
        \le & C \left(\bar \rho^{\frac{3}{4}(1-\alpha)} + \bar \rho^{-2\alpha+\frac 12 \gamma+\frac 32} \right),
    \end{aligned}
\end{equation}
where we have used \eqref{a1}.
Using the above estimates, and integrating \eqref{t1} with respect to $t$ over $[0, T]$, we can get from \eqref{basic-est} and \eqref{d2d} that
\begin{equation}
        \begin{aligned}
            & \frac{\bar \rho^\alpha}{2^{\alpha+1}}\sup_{0\le t\le T}\int \left(  \mu_1 |\nabla u|^2+ (\mu_1+\mu_2)(\dv u)^2\right)dx + \frac{1}{2} \int_{0}^{T}\|\sqrt{\rho} u_t\|_{L^2}^2 dt \\
            \le
            &\sup_{0\le t\le T}\int \left(  \mu_1 \rho^\alpha |\mathcal{D} u|^2+ \frac{\mu_2 \rho^\alpha}{2}(\dv u)^2\right)dx + \frac{1}{2} \int_{0}^{T}\|\sqrt{\rho} u_t\|_{L^2}^2 dt \\
            \le&  N_3 2^\alpha \bar \rho^\alpha   + C \sup_{0\le  t\le T}  \bar \rho^{\gamma-1} \|\rho-\bar \rho\|_{L^2}\|\dv u\|_{L^2}+ C  \sup_{0\le  t\le T} \|\nabla d\|_{L^3} \|\nabla^2 d\|_{L^2}   \|\dv u\|_{L^2} \\
            %&+ C  \int_{0}^{T} \left( \bar \rho^{(\gamma-1)\frac{q+6}{2q}-(\alpha-1)\frac{q-6}{2q} } \|\nabla \rho\|_{L^q}^{\frac{5(q+6)}{4(4q-3)}} \|\nabla u\|_{L^2}^{\frac{3}{2}}+  \bar \rho^{\frac{3\gamma-\alpha}{2} } \|\nabla \rho\|_{L^q}^{\frac{15q}{4(4q-3)}} \|\nabla u\|_{L^2}^{\frac{5q-6}{2q}}\right) dt  \\
            & + C \bar \rho^\alpha \int_0^{T}\|\nabla u\|_{L^2}^{\frac{3}{2}} \|\nabla u\|_{L^6}^{\frac{3}{2}} dt + C(\bar \rho + \bar{\rho}^{\alpha}) \int_{0}^{T} \|\nabla u\|_{L^2}^{2} \|\nabla u\|_{L^6} dt \\
            &+  C  (\bar \rho^{\gamma}+\bar \rho) \int_{0}^{T} \|\nabla u\|_{L^2}^{2} dt + C\int_{0}^{T} \|\nabla^3d\|_{L^2}^2 dt\\
            \le&  N_3 2^\alpha \bar \rho^\alpha   +  C (\bar \rho^{\frac{\gamma+1}{2}} +1 +\bar\rho+ \bar \rho^{\alpha+ \frac{3}{4}(1-\alpha)} + \bar \rho^{-\alpha+\frac 12 \gamma+\frac 32} + \bar \rho^{\gamma+1-\alpha})\\
            \le&  N_3 2^\alpha \bar \rho^\alpha   +  C_1 ( \bar \rho^{\frac{\gamma+1}{2}}  + \bar \rho^{\alpha+ \frac{3}{4}(1-\alpha)} + \bar \rho^{\gamma+1-\alpha})\\
            \le& N_3 2^{\alpha+1} \bar \rho^\alpha,
        \end{aligned}
    \end{equation}
    % where we have used 
    % \begin{align*}
    %     &\int_{0}^{T} \left( \bar \rho^{(\gamma-1)\frac{q+6}{2q}-(\alpha-1)\frac{q-6}{2q} } \|\nabla \rho\|_{L^q}^{\frac{5(q+6)}{4(4q-3)}} \|\nabla u\|_{L^2}^{\frac{3}{2}}+  \bar \rho^{\frac{3\gamma-\alpha}{2} } \|\nabla \rho\|_{L^q}^{\frac{15q}{4(4q-3)}} \|\nabla u\|_{L^2}^{\frac{5q-6}{2q}}\right) dt\\
    %     &\le 
    % \end{align*}
    after defining 
    \begin{equation}\label{m1}
    N_3 \triangleq 2\mu_1 \|\mathcal{D}(u_0) \|_{L^2}^2+\mu_2 \|\mathrm{div}u_0\|_{L^2}^2, 
    \end{equation}
    and choosing $\Lambda_2$ such that 
    $$\bar \rho \ge \Lambda_2\triangleq \max\left\{1, \left(\frac{4C_1}{2^\alpha N_3}\right)^{\frac{2}{2\alpha-\gamma-1}}, \left(\frac{4C_1}{2^\alpha N_3}\right)^{\frac{4}{3(\alpha-1)}}, \left(\frac{4C_1}{2^\alpha N_3}\right)^{\frac{1}{2\alpha-\gamma-1}}\right\},$$
    as long as $\alpha>\max\left\{1,\frac{\gamma+1}{2}\right\}$.
\end{proof}

Next we deal with $\mathcal{E}_{u,2}$.
\begin{lemma}
    Under the assumption \eqref{a1} and 
    \begin{equation}
        \alpha>\gamma-1,
    \end{equation}
    then there exists a positive constant $\Lambda_3$ such that 
\begin{equation}
\mathcal{E}_{u,2}(T)\le 2N_4 \bar\rho^{2\alpha-1},
\end{equation}
% and 
% \begin{equation}\label{trut}
%     \sup_{0\le t\le T} t \left(\|\sqrt{\rho} u_t \|_{L^2}^2
%     +\|\nabla d_t\|_{L^2}^2\right) 
%     +\int_{0}^{T} t\left(\bar\rho^\alpha\|\nabla u_t\|^2_{L^2}
%     +\|\nabla^2 d_t\|_{L^2}^2\right) dt\\
%     \leq C\bar\rho^\alpha,
% \end{equation}
provided $\bar \rho\ge \Lambda_3= \Lambda_3(\mu_1,\mu_2,\nu, \lambda, \alpha,a, \|\nabla \rho_0\|_{L^q}, \| u_0\|_{H^1}, \|\nabla^2 d_0\|_{L^2})$.
% and
% \begin{equation}\label{ttrut}
%     \sup_{0\le t\le T} t^2 \left(\|\sqrt{\rho} u_t \|_{L^2}^2
%     +\|\nabla d_t\|_{L^2}^2\right) 
%     +\int_{0}^{T} t^2\left(\bar\rho^\alpha\|\nabla u_t\|^2_{L^2}
%     +\|\nabla^2 d_t\|_{L^2}^2\right) dt
%     \leq C\bar\rho.
% \end{equation}
\end{lemma}

\begin{proof}
Multiply the resulting equation by $u_t$, and after integrating by parts, we have that 
\begin{equation}\label{k2}
    \begin{aligned}
        &\frac{1}{2} \frac{d}{dt} \int \rho |u_t|^2 dx + \int  \left(  2 \mu_1 \rho^\alpha |\mathcal{D} u_t|^2+ \mu_2 \rho^\alpha (\dv u_t)^2\right)dx\\
            =& \int  \dv(\rho u) |u_t|^2 dx -\int  \rho u_t \cdot \nabla u \cdot u_t dx + \int  \dv(\rho u) u \cdot \nabla u \cdot u_t dx  \\
            &+ \int \left( 2 \mu_1 (\rho^\alpha)_t \mathcal{D}u: \mathcal{D}u_t+ \mu_2 (\rho^\alpha)_t \dv u \dv u_t \right)dx- \int P_t \dv u_t dx \\
            &+  \nu \int (2\mathcal{D}_{ij}(u_t)\partial_id\cdot\partial_jd_t + \nabla d : \nabla d_t \dv u_t )dx\\
            =&\sum_{i=1}^{6} I_i.
    \end{aligned}
\end{equation}

It follows from H\"older and Sobolev inequalities that
\begin{equation}
    \begin{aligned}
        I_1+I_2\le& C \int  (\bar\rho| \nabla u| +| \nabla \rho \cdot u|) |u_t|^2 dx  \\
            \le &C \left( \bar \rho  \|\nabla u\|_{L^2} +  \|\nabla \rho\|_{L^q} \|u \|_{L^{\frac{2q}{q-2}}}  \right) \| u_t \|_{L^4}^{2} \\
            \le &C \left( \bar \rho^{\frac34}  \|\nabla u\|_{L^2} +  \bar \rho^{-\frac14} \|\nabla \rho\|_{L^q} \|u \|_{L^{2}}^{\frac{q-3}{q}} \|\nabla u \|_{L^{2}}^{\frac{3}{q}} \right) \|\sqrt{\rho} u_t \|_{L^2}^{\frac12} \|\nabla u_t \|_{L^2}^{\frac32} \\
            \le & \frac{\mu}{2^{\alpha+3}} \bar \rho^{\alpha} \|\nabla u_t\|_{L^2}^2 + C\left(\bar \rho^{3-3\alpha} \|\nabla u\|_{L^2}^4  +   \bar \rho^{-3\alpha-1} \|\nabla \rho\|_{L^q}^4  \|\nabla u \|_{L^{2}}^{\frac{12}{q}}  \right) \|\sqrt{\rho} u_t \|_{L^2}^2 \\
            \le & \frac{\mu}{2^{\alpha+3}} \bar \rho^{\alpha} \|\nabla u_t\|_{L^2}^2 + C\left(\bar \rho^{3-3\alpha} \|\nabla u\|_{L^2}^2  +   \bar \rho^{-3\alpha-1+2b}   \|\nabla u \|_{L^{2}}^{2}  \right) \|\sqrt{\rho} u_t \|_{L^2}^2\\
            \le & \frac{\mu}{2^{\alpha+3}} \bar \rho^{\alpha} \|\nabla u_t\|_{L^2}^2 + C \bar \rho^{2-\alpha} \|\nabla u\|_{L^2}^2, 
    \end{aligned}
\end{equation}
% \begin{equation}
%     \begin{aligned}
%         J_1\leq& C t \int |\rho \nabla u|  |u_t|^2 dx  \\
%         \le &C t\bar \rho^{\frac12}\|\sqrt{\rho} u_t \|_{L^2} \|\nabla u\|_{L^3}\| u_t \|_{L^6} \\
%         \le &Ct  \bar\rho^{\frac12}\|\sqrt{\rho} u_t \|_{L^2} \|\nabla u\|_{L^2}^{\frac12} \|\nabla u\|_{L^6}^{\frac12} \| \nabla u_t \|_{L^2}\\
%         \le & \frac{\mu}{8} t \bar\rho^{\alpha} \|\nabla u_t\|_{L^2}^2 
%         +Ct\bar \rho^{1-\alpha}\|\sqrt{\rho} u_t \|_{L^2}^2 \|\nabla u\|_{L^2} \|\nabla u\|_{L^6}\\
%         \leq & \frac{\mu}{8} t \bar\rho^{\alpha} \|\nabla u_t\|_{L^2}^2 
%         +Ct\bar\rho^{-\alpha} \|\nabla u\|_{L^2} \|\sqrt{\rho} u_t \|_{L^2}^2\|\nabla u\|_{H^1}\\
%         \leq & \frac{\mu}{8} t \bar\rho^{\alpha} \|\nabla u_t\|_{L^2}^2
%         +Ct\bar\rho^{-\alpha} \|\nabla u\|_{L^2} \|\sqrt{\rho} u_t \|_{L^2}^2\\
%         &\cdot(\bar\rho^{\frac{1}{2}-\alpha}\|\sqrt{\rho} u_t\|_{L^2}
%         +\bar\rho^{2-2\alpha}\|\nabla u\|_{L^2}^3\color{red}{+\bar\rho^{-\alpha}\|\nabla d\|_{L^3}\|\nabla^3d\|_{L^2}})\\
%         \leq & \frac{\mu}{8} t \bar\rho^{\alpha} \|\nabla u_t\|_{L^2}^2
%         +Ct\bar\rho^{\frac{1}{2}-2\alpha} \|\nabla u\|_{L^2} \|\sqrt{\rho} u_t \|_{L^2}^3
%         +Ct\bar\rho^{2-3\alpha} \|\nabla u\|_{L^2}^4 \|\sqrt{\rho} u_t \|_{L^2}^2\\
%         &\color{red}{\quad+Ct\bar\rho^{-2\alpha}\|\nabla u\|_{L^2}^2\|\sqrt{\rho}u_t\|_{L^2}^4
%         +Ct\bar\rho^{-2\alpha}\|\nabla d\|_{L^3}^2\|\nabla^3d\|_{L^2}^2},
%     \end{aligned}
% \end{equation}
also $I_3$ can be estimated as follows
\begin{equation}
    \begin{aligned}
        I_3\le & \int  |\rho \dv u+ \nabla \rho \cdot u| |u| | \nabla u| |u_t| dx\\
            \le & C \left(\bar \rho \|u\|_{L^6} \|\nabla u\|_{L^3}^2  \| u_t\|_{L^6} + \|\nabla \rho\|_{L^q}  \|u\|_{L^6}^2  \|\nabla u\|_{L^{\frac{2q}{q-2}}} \|u_t\|_{L^6} \right)\\
            \le & \frac{\mu}{2^{\alpha+3}} \bar \rho^{\alpha} \|\nabla u_t\|_{L^2}^2 + C \bar \rho^{2-\alpha} \|\nabla u\|_{L^2}^4 \|\nabla u\|_{L^6}^2 + C \bar \rho^{ - \alpha} \|\nabla \rho \|_{L^q}^2  \|\nabla u\|_{L^2}^\frac{6(q-1)}{q}\|\nabla u\|_{L^6}^\frac{6}{q}.
    \end{aligned}
\end{equation}
Similarly, we can estimate $I_4$ as follows:
    \begin{equation}
        \begin{aligned}
            I_4\le & C  \bar \rho^{\alpha-1}\int |\rho \dv u+ \nabla \rho \cdot u| |\nabla  u||\nabla  u_t| dx\\
            \le & C \left(\bar \rho^{\alpha} \|\nabla u\|_{L^4}^2 \|\nabla  u_t\|_{L^2} + \bar \rho^{\alpha-1}  \|\nabla \rho\|_{L^q}  \|u\|_{L^\frac{3q}{q-3}}  \|\nabla u\|_{L^6} \|\nabla u_t\|_{L^2} \right)\\
            \le & \frac{\mu}{2^{\alpha+3}} \bar \rho^{\alpha} \|\nabla u_t\|_{L^2}^2 + C \bar \rho^{\alpha} \|\nabla u\|_{L^2} \|\nabla u\|_{L^6}^3 + C \bar \rho^{\alpha-2} \|\nabla \rho \|_{L^q}^2  \|\nabla u\|_{L^2}^{\frac{3(q-2)}{q}} \|\nabla u\|_{L^6}^{2+\frac{6-q}{q}},
        \end{aligned}
    \end{equation}
and $I_5$ is bounded by
    \begin{equation}
        \begin{aligned}
            I_5\le& C  \bar \rho^{\gamma-1}\int |\rho \dv u+ \nabla \rho \cdot u| |\dv u_t| dx \\
            \le & C \left(\bar \rho^{\gamma} \|\dv u\|_{L^2} \|\dv u_t\|_{L^2} + \bar \rho^{\gamma-1}  \|\nabla \rho\|_{L^q}  \|u\|_{L^{\frac{2q}{q-2}}} \|\dv u_t\|_{L^2} \right)\\
            \le & \frac{\mu}{2^{\alpha+3}} \bar \rho^{\alpha} \|\nabla u_t\|_{L^2}^2 + C \bar \rho^{2\gamma-\alpha} \|\nabla u\|_{L^2}^2 + C \bar \rho^{2\gamma-2-\alpha} \|\nabla \rho \|_{L^q}^2 \|u\|_{L^2}^{\frac{2(q-3)}{q}} \|\nabla u\|_{L^2}^{\frac{6}{q}}\\
            \le & \frac{\mu}{2^{\alpha+3}} \bar \rho^{\alpha} \|\nabla u_t\|_{L^2}^2 + C \bar \rho^{2\gamma-\alpha} \|\nabla u\|_{L^2}^2 + C \bar \rho^{2\gamma-2-\alpha}  \|\nabla \rho\|_{L^q}^{2}.
            %\\ \le & \frac{\mu}{2^{\alpha+3}} \bar \rho^{\alpha} \|\nabla u_t\|_{L^2}^2 + C \bar \rho^{2\gamma-\alpha} \|\nabla u\|_{L^2}^2 . 
        \end{aligned}
    \end{equation}
    Then,
\begin{align}\label{J5}
    I_6&\leq C\int |\nabla u_t| |\nabla d||\nabla  d_t| dx
    \leq C\|\nabla u_t\|_{L^2}\|\nabla d\|_{L^3}\|\nabla d_t\|_{L^6}\nonumber\\
    &\leq \frac{1}{16}\bar \rho^{\alpha} \|\nabla u_t\|_{L^2}^2  
    +C\bar\rho^{-\alpha}\|\nabla d\|_{L^3}^2\| \nabla^2 d_t\|_{L^2}^2.
\end{align}
Collecting all above estimates, we have 
\begin{equation}
    \begin{aligned}
        & \frac{d}{dt} \int \rho |u_t|^2 dx + \mu\bar \rho^\alpha \int    |\nabla u_t|^2dx\\
            \le & C \bar \rho^{2-\alpha} \|\nabla u\|_{L^2}^2 + C \bar \rho^{2-\alpha} \|\nabla u\|_{L^2}^4 \|\nabla u\|_{L^6}^2 + C \bar \rho^{ - \alpha} \|\nabla \rho \|_{L^q}^2  \|\nabla u\|_{L^2}^\frac{6(q-1)}{q}\|\nabla u\|_{L^6}^\frac{6}{q}\\
            & + C \bar \rho^{\alpha} \|\nabla u\|_{L^2} \|\nabla u\|_{L^6}^3 + C \bar \rho^{\alpha-2} \|\nabla \rho \|_{L^q}^2  \|\nabla u\|_{L^2}^{\frac{3(q-2)}{q}} \|\nabla u\|_{L^6}^{2+\frac{6-q}{q}} \\
            &+ C \bar \rho^{2\gamma-\alpha} \|\nabla u\|_{L^2}^2 + C \bar \rho^{2\gamma-2-\alpha}  \|\nabla \rho\|_{L^q}^{2}+ C\bar\rho^{-\alpha}\|\nabla d\|_{L^3}^2\| \nabla^2 d_t\|_{L^2}^2\\
            \le & C (\bar \rho^{2-\alpha} + \bar \rho^{2\gamma-\alpha} ) \|\nabla u\|_{L^2}^2 + C\bar\rho^{-\alpha}\|\nabla d\|_{L^3}^2\| \nabla^2 d_t\|_{L^2}^2\\
            & + C \bar \rho^{\alpha} \|\nabla u\|_{L^2} \|\nabla u\|_{L^6}^3 + C \bar \rho^{\alpha}  \|\nabla u\|_{L^2}^{3} \|\nabla u\|_{L^6},
    \end{aligned}
\end{equation}
due to $q\in(3,6) $.

Operating $\partial_t$ to $\eqref{ins}_3$, and multiplying it by $-\Delta d_t$, by integrating the resulting equality
by parts, we get
\begin{equation}\label{ddt}
    \begin{aligned}
    &\frac{1}{2}\frac{d}{dt}\|\nabla d_t\|_{L^2}^2
    +\lambda \|\nabla^2 d_t\|_{L^2}^2\\
    &=\int u_t\cdot\nabla d\cdot \Delta d_t dx -\lambda \int(|\nabla d|^2d)_t\cdot \Delta d_t dx\\
    &\quad-\int \nabla u:\nabla d_t\odot\nabla d_t dx -\frac 12\int \dv u |\nabla d_t|^2 dx\\
    &\leq C (\|u_t\|_{L^6}\|\nabla d\|_{L^3}+\|\nabla d_t\|_{L^6}\|\nabla d\|_{L^3}+\|\nabla d\|_{L^6}^2\|d_t\|_{L^6})\|\Delta d_t\|_{L^2}\\
    &\quad+C\|\nabla u\|_{L^2}\|\nabla d_t\|_{L^4}^2\\
    &\leq C (\|\nabla u_t\|_{L^2}\|\nabla d\|_{L^3}+\|\nabla^2 d_t\|_{L^2}\|\nabla d\|_{L^3}+\|\nabla^2 d\|_{L^2}^2\|\nabla d_t\|_{L^2})\|\Delta d_t\|_{L^2}
    \\
    &\quad +C\|\nabla u\|_{L^2}\|\nabla d_t\|_{L^2}^{\frac{1}{2}}\|\nabla^2d_t\|_{L^2}^\frac{3}{2}\\
    &\leq \frac{\lambda }{4}\| \nabla^2 d_{t}\|_{L^2}^2+C\|\nabla d\|_{L^3}^2\|\nabla u_t\|_{L^2}^2
    +c_4 \lambda \|\nabla d\|_{L^3}^2\|\nabla^2 d_t\|_{L^2}^2 \\
    &\quad+C(\|\nabla u\|_{L^2}^4+\|\nabla^2 d\|_{L^2}^4)\|\nabla d_t\|_{L^2}^2,
\end{aligned}
\end{equation}
where we have used that 
\begin{align*}
    \int u\cdot\nabla d_t\cdot\Delta d_t dx&=-\int\partial_j u_i\partial_id_t\cdot\partial_j d_t dx-\int (u\cdot \nabla) \partial_j d_t\cdot \partial_j d_t dx\\
    &=-\int \nabla u:\nabla d_t\odot\nabla d_t dx-\frac 12\int \dv u |\nabla d_t|^2 dx.
\end{align*}
% and
% \begin{align*}
%     \|d_t\|_{L^2}^2&\leq C(\|\Delta d\|_{L^2}^2+\|u\cdot\nabla d\|_{L^2}^2+\|\nabla d\|_{L^4}^4)\\
%     &\leq C\|\nabla ^2 d\|_{L^2}^2+C\|\nabla d\|_{L^3}^2(\|\nabla u\|_{L^2}^2+\|\nabla^2d\|_{L^2}^2).
% \end{align*}
Choosing 
$$\delta<\frac{1}{4c_4},$$ 
we have
\begin{equation}\label{d_t-2nd}
    \begin{aligned}
      \frac{d}{dt}\|\nabla d_t\|_{L^2}^2
    +\|\nabla^2 d_{t}\|_{L^2}^2\leq &C\|\nabla d\|_{L^3}^2\|\nabla u_t\|_{L^2}^2+C(\|\nabla u\|_{L^2}^4+\|\nabla^2 d\|_{L^2}^4)\|\nabla d_t\|_{L^2}^2.   
    \end{aligned}
\end{equation}

Combining all the above estimates \eqref{k2}–\eqref{J5}, \eqref{d_t-2nd}
\eqref{a1} and \eqref{basic-est}, we deduce
\begin{equation}\label{k33}
    \begin{aligned}
        &\frac{d}{dt} \int \left(\rho |u_t|^2+ \vert\nabla d_t\vert^2\right) dx 
        +  \int\left(\bar\rho^\alpha|\nabla u_t|^2 + |\nabla^2 d_t|^2\right)dx\\
        \leq &C (\bar \rho^{2-\alpha} + \bar \rho^{2\gamma-\alpha} ) \|\nabla u\|_{L^2}^2 + C\bar\rho^{-\alpha} \|\nabla d\|_{L^3}^2 (\| \nabla^2 d_t\|_{L^2}^2 +\bar \rho^{\alpha}\|\nabla u_t\|_{L^2}^2 )\\
            & + C \bar \rho^{\alpha} \|\nabla u\|_{L^2} \|\nabla u\|_{L^6}^3 + C \bar \rho^{\alpha}  \|\nabla u\|_{L^2}^{3} \|\nabla u\|_{L^6}\\
        & +C(\|\nabla u\|_{L^2}^4+\|\nabla^2 d\|_{L^2}^4)\|\nabla d_t\|_{L^2}^2\\
        \leq &C (\bar \rho^{2-\alpha} + \bar \rho^{2\gamma-\alpha} ) \|\nabla u\|_{L^2}^2 + c_5 \bar\rho^{-\alpha}  (\| \nabla^2 d_t\|_{L^2}^2 +\bar \rho^{\alpha}\|\nabla u_t\|_{L^2}^2 )\\
            & + C \bar \rho^{\alpha} \|\nabla u\|_{L^2} \|\nabla u\|_{L^6}^3 + C \bar \rho^{\alpha}  \|\nabla u\|_{L^2}^{3} \|\nabla u\|_{L^6}\\
        & +C(\|\nabla u\|_{L^2}^4+\|\nabla^2 d\|_{L^2}^4)\|\nabla d_t\|_{L^2}^2  ,
    \end{aligned}
\end{equation}
due to $q\in (3,6)$, \eqref{a1}.

Therefore, choosing 
$$\bar \rho > (4c_{5})^\frac{1}{{\alpha}},$$
Gronwall's inequality yields
\begin{equation}
    \begin{aligned}
        &\sup_{0\le t\le T}  \left(\|\sqrt{\rho} u_t \|_{L^2}^2
        + \|\nabla d_t\|_{L^2}^2\right) 
        +\int_{0}^{T} \left(\bar\rho^\alpha\|\nabla u_t\|^2_{L^2}
        + \|\Delta d_t\|_{L^2}^2\right) dt\\
        \leq & \bigg\{C\int_0^T ( \bar \rho^{2-\alpha} + \bar \rho^{2\gamma-\alpha}) 
        \|\nabla u\|_{L^2}^2  dt \\
        %+C\int_0^T \left(\|\sqrt{\rho} u_t\|_{L^2}^2 + \|\nabla d_t\|_{L^2}^2 \right)dt\\
        &+C  \bar \rho^{\alpha} \int_0^T (\|\nabla u\|_{L^2} \|\nabla u\|_{L^6}^3 +  \|\nabla u\|_{L^2}^{3} \|\nabla u\|_{L^6} ) dt \bigg\} \\
        &\cdot\exp{\left\{ C \int_0^T \|\nabla u\|_{L^2}^2 +\|\nabla^2 d\|_{L^2}^4  dt\right\}}.
    \end{aligned}
\end{equation}

Taking advantage of \eqref{a1}, \eqref{basic-est}, and \eqref{e1}, we obtain
\begin{equation}\label{k3}
    \begin{aligned}
        &\int_{0}^{T} \|\nabla u\|_{L^2} \|\nabla u\|_{L^6}^{3}  dt\\
        \le & C\int_{0}^{T} \|\nabla u\|_{L^2} \left(\bar\rho^{\frac{1}{2}-\alpha} \|\sqrt{\rho} u_t\|_{L^2} +\|\nabla u\|_{L^2} + \bar \rho^{-\alpha+\gamma-1}\|\nabla \rho\|_{L^2}  \right)^{3} dt\\
        & + C\int_{0}^{T} \|\nabla u\|_{L^2} \left(\bar\rho^{-\alpha}\|\nabla^2 d\|_{L^2}^{\frac 32}\|\nabla^3 d\|_{L^2}^{\frac 12} \right)^{3} dt\\
        \le & C \int_{0}^{T} \left(\bar \rho^{1-2\alpha} \|\sqrt{\rho} u_t\|_{L^2}^2 + \|\nabla u\|_{L^2}^2 + \bar \rho^{-3\alpha+3\gamma-3 +\beta } \|\nabla u\|_{L^2} \|\nabla \rho\|_{L^2}  \right) dt\\
        & + C \int_{0}^{T}  \bar \rho^{-3\alpha} \|\nabla u\|_{L^2}^{\frac 12}\|\nabla^3d\|_{L^2}^{\frac{3}{2}}  dt\\
        \le & C \left(\bar \rho^{1-\alpha}+ \bar \rho^{-3\alpha+3\gamma-3 +\beta  + (1-\alpha) \frac{1}{2} + \frac 12 (\alpha-\gamma + \beta)} + \bar \rho^{-3\alpha+ \frac 14(1-\alpha)}\right) \\
        \le & C \left(\bar \rho^{1-\alpha}+ \bar \rho^{-3\alpha + \frac 32 \beta  + \frac 52 (\gamma-1 )} + \bar \rho^{-3\alpha+ \frac 14(1-\alpha)}\right) \\
        \le & 
        \begin{cases}
            C (\bar \rho^{1-\alpha} + \bar \rho^{-3\alpha+\frac 52 (\gamma-1)  }  ), \quad \gamma>3,\\
            C (\bar \rho^{1-\alpha}  + \bar \rho^{-3\alpha + 2+\gamma} ), \quad 1<\gamma\le 3,
        \end{cases}\\
        \le &  
        \begin{cases}
            C (\bar \rho^{1-\alpha} + \bar \rho^{-\frac 12 \alpha-5  }  ), \quad \gamma>3,\\
            C (\bar \rho^{1-\alpha}  + \bar \rho^{-3\alpha + 2+\gamma} ), \quad 1<\gamma\le 3,
        \end{cases}\\
        \le &  C (\bar \rho^{\frac12(1-\alpha)}  + \bar \rho^{-3\alpha + 2+\gamma} ),
    \end{aligned}
\end{equation}
due to $\gamma\le \alpha-1$.
Substitute \eqref{k1} and \eqref{k3} into \eqref{k33}, one gets
\begin{equation}\label{trut1}
\begin{aligned}
    &\sup_{0\le t\le T}  \left(\|\sqrt{\rho} u_t \|_{L^2}^2
    + \|\nabla d_t\|_{L^2}^2\right) 
    +\int_{0}^{T} \left(\bar\rho^\alpha\|\nabla u_t\|^2_{L^2}
    + \|\Delta d_t\|_{L^2}^2\right) dt\\
    &\leq \left(\|\sqrt{\rho} u_t \|_{L^2}^2+ \|\nabla d_t\|_{L^2}^2\right)|_{t=0} \\
    &\quad+  C(\bar\rho^{3-2\alpha}+\bar\rho^{2\gamma-2\alpha+1} + \bar \rho + \bar \rho^{\gamma-\alpha+1} + \bar \rho^{\frac12(1+\alpha)}  + \bar \rho^{-2\alpha + 2+\gamma} )\cdot
    \exp{\{C\bar\rho^{1-\alpha}+\tilde{C}\}}.
\end{aligned}
\end{equation}

Define 
    \begin{align*}
        \rho_0^{1/2} u_{t0} \triangleq \rho_0^{-1/2}\left(-\rho_0 u_0\cdot \nabla u_0 - \nabla P(\rho_0) +2\mu \dv( \rho_0^{\alpha}\mathcal{D} u_0)+\lambda\nabla( \rho_0^{\alpha} \dv u_0)\right)\\
        -\nu \rho_0^{-1/2} \left(\mathrm{div}(\nabla d_0\odot\nabla d_0)-\frac{1}{2} |\nabla d_0|^2 \mathbb{I}_3)\right),
    \end{align*}
    and 
    \begin{equation*}
    \nabla d_{t0}\triangleq \lambda \nabla\Delta d_0-\nabla(u_0\cdot\nabla d_0)+ \lambda \nabla(\vert\nabla d_0\vert^2 d_0),
    \end{equation*}
    then it follows from \eqref{ia} that 
    \begin{equation}\label{ki}
    \begin{aligned}
        \|\sqrt{\rho_0} u_{t0}\|_{L^2} \le& C 
        \left(\bar \rho^{\frac12} \|\nabla u_0\|_{L^2}^{\frac{3}{2}} \|\nabla u_0\|_{L^6}^{\frac{1}{2}} + \bar \rho^{-\frac{1}{2}}\|\nabla P(\rho_0)\|_{L^2} + \bar \rho^{\alpha-\frac{1}{2}}\|\nabla^2 u_0\|_{L^2} \right) \\
        &+ C \left( \bar \rho^{\alpha-\frac{3}{2}} \|\nabla \rho_0\|_{L^q} \|\nabla u_0\|_{L^{\frac{2q}{q-2}}} + C\bar \rho^{-\frac12} \|\nabla d_0\|_{L^6}\|\nabla^2 d_0\|_{L^3} \right)\\
        \le& C  \left(\bar \rho^{\gamma-\frac{3}{2}} \|\nabla \rho_0\|_{L^2}+ \bar \rho^{\alpha-\frac{1}{2}}\right)\\
        \le & C\left(\bar \rho^{\gamma-\frac 32}+ \bar \rho^{\alpha-\frac{1}{2}}\right)\\
        \le & C_2 \bar \rho^{\alpha-\frac{1}{2}},
    \end{aligned}
    \end{equation}
    due to 
    $\gamma-1\le \alpha$, and 
    \begin{equation}
    \begin{aligned}
        \|\nabla d_{t0}\|_{L^2} \le& C 
        \left( \| \nabla \Delta  d_0\|_{L^2} + \|\nabla u_0\|_{L^6} \|\nabla d_0\|_{L^3} + \| u_0\|_{L^6} \|\nabla^2 d_0\|_{L^3}\right) \\
        &+ C \left( \|\nabla d_0\|_{L^6}^3 +  \|\nabla^2 d_0\|_{L^3} \|\nabla d_0\|_{L^{6}}  \right)\\
        \le & C_3.
    \end{aligned}
    \end{equation}
    By combining \eqref{k3} with \eqref{ki}, we obtain
    \begin{equation}
        \begin{aligned}
            &\sup_{0\le t\le T}  \|\sqrt{\rho} u_t \|_{L^2}^2 +   \mu \frac{\bar \rho^\alpha}{2^{\alpha+1}} \int_{0}^{T} \|\nabla u_t\|^2_{L^2} dt\\
            \le &  N_4 \bar \rho^{2\alpha-1}  +  C_4 \left( \bar \rho^{\alpha} + \bar \rho^{2\gamma-2\alpha +1} + \bar \rho^{\gamma} \right) \\
            \le & 2 N_4 \bar \rho^{2\alpha-1},  
        \end{aligned}
    \end{equation}
    after defining
    \begin{equation}\label{m2}
       N_4=N_4(a,\mu_1,\mu_2,\lambda, \alpha, \gamma, \|\rho_0-\bar \rho\|_{L^\gamma}, \|\nabla \rho_0\|_{L^q}, \|\nabla u_0\|_{H^1}, \|\nabla d_0\|_{H^2}) \triangleq C_2^2 + C_3^2, 
    \end{equation}
    and choosing 
    $$\bar \rho \ge \Lambda_3 \triangleq \max\left\{\Lambda_2, \left(\frac{3C_4}{ N_4}\right)^{\frac{1}{\alpha-1}}, \left(\frac{3C_4}{ N_4}\right)^{\frac{1}{2(2\alpha-\gamma-1)}}, \left(\frac{3C_4}{ N_4}\right)^{\frac{1}{2\alpha-\gamma-1}}, (4c_{5})^\frac{1}{{\alpha}} \right\},$$
    as long as $\alpha>\max\left\{1,\frac{\gamma+1}{2}\right\}$.
\end{proof}
Finally, we are able to finish the bound of  $\mathcal{E}_\rho$.
\begin{lemma}\label{L_1}
	There exists a positive constant $\Lambda_4$ such that 
	\begin{equation}
		\mathcal{E}_{\rho,2}(T) \le N_1 \bar\rho^{\beta},
        ~\mathcal{E}_{\rho,3}(T) \le N_2 \bar\rho^{\beta},
	\end{equation}
	provided $\bar\rho>\Lambda_4=\Lambda_4(\mu_1,\mu_2,\nu,\lambda, \alpha,a,\|\nabla\rho_0\|_{L^q}, \| u_0\|_{H^2}, \|\nabla d_0\|_{H^1})$.
    As a consequence, 
    \begin{equation}
		\mathcal{E}_{\rho,1}(T) \le \frac{\bar \rho}{4}.
	\end{equation}
\end{lemma}
\begin{proof}
	It follows from \eqref{ins}$_1$ and \eqref{flux} that  
    \begin{equation}
        \nabla \rho_t + u \cdot \nabla^2 \rho + \nabla u \cdot \nabla \rho + \nabla \rho \dv u+ \frac{1}{2\mu+\lambda} \rho \left(\nabla F + \nabla \mathcal{P} \right)=0.
    \end{equation}
    Multiplying the above equation by $|\nabla \rho|^{q-2}\nabla \rho$ and then integrating by parts, we have 
    \begin{equation}
        \begin{aligned}
            &\frac{1}{p}\frac{d}{dt} \|\nabla \rho\|_{L^q}^q + \frac{a \gamma }{2\mu+\lambda}\int  \rho^{\gamma-\alpha}|\nabla \rho|^qdx \\
            =& \frac{1-q}{q}\int \dv u |\nabla \rho |^qdx - \int \nabla \rho \cdot \nabla u \cdot \nabla \rho |\nabla \rho|^{q-2} dx - \frac{ 1 }{2\mu+\lambda} \int \rho \nabla F \cdot \nabla \rho |\nabla \rho|^{p-2} dx \\
            \le & C \|\nabla u\|_{L^\infty} \|\nabla \rho \|_{L^q}^q + C \bar \rho \|\nabla F\|_{L^q} \|\nabla \rho \|_{L^q}^{q-1}.
        \end{aligned}
    \end{equation}
    Then, we have
    \begin{equation}
        \frac{d}{dt} \|\nabla \rho\|_{L^q}^2 + \bar \rho^{\gamma-\alpha}   \|\nabla \rho\|_{L^q}^2  \le C \|\nabla u\|_{L^\infty} \|\nabla \rho \|_{L^q}^2 + C \bar \rho \|\nabla F\|_{L^q} \|\nabla \rho \|_{L^q},
    \end{equation}
    which together with \eqref{e3} and \eqref{W2p} gives that 
    \begin{equation}
    \begin{aligned}
        &\frac{d}{dt} \|\nabla \rho\|_{L^q}^2 + \bar \rho^{\gamma-\alpha}   \|\nabla \rho\|_{L^q}^2\\
        \le& C \|\nabla u\|_{L^\infty} \|\nabla \rho \|_{L^q}^2 + C \bar\rho^{-\alpha+1} \|\rho u_t\|_{L^q} \|\nabla \rho \|_{L^q} + C \bar\rho^{-\alpha+2}  \| u\|_{L^q} \| \nabla u\|_{L^\infty} \|\nabla \rho \|_{L^q} \\
        &+ C \bar \rho^{1-\alpha} \|\nabla^2 d\|_{L^2}^{\frac{3}{q}}\|\nabla^3 d\|_{L^2}^{2-\frac{3}{q}} \|\nabla \rho\|_{L^q}\\
        \le& C \bar\rho \|\nabla u\|_{L^\infty} \|\nabla \rho \|_{L^q} + C \bar\rho^{-\alpha+1} \|\rho u_t\|_{L^q} \|\nabla \rho \|_{L^q}  + C \bar \rho^{1-\alpha} \|\nabla^2 d\|_{L^2}^{\frac{3}{q}}\|\nabla^3 d\|_{L^2}^{2-\frac{3}{q}} \|\nabla \rho\|_{L^q}\\
        \le& C \bar\rho \left(\bar\rho^{-\alpha}\|\rho u_t\|_{L^q}  + \bar \rho^{\gamma-\alpha-1}\|\nabla \rho \|_{L^q} + \bar\rho^{-\alpha} \|\nabla^2 d\|_{L^2}^{\frac{3}{q}}\|\nabla^3 d\|_{L^2}^{2-\frac{3}{q}}+\|\nabla u\|_{L^2} \right) \|\nabla \rho \|_{L^q} \\
        &+ C \bar\rho^{-\alpha+1} \|\rho u_t\|_{L^q} \|\nabla \rho \|_{L^q}  + C \bar \rho^{1-\alpha} \|\nabla^2 d\|_{L^2}^{\frac{3}{q}}\|\nabla^3 d\|_{L^2}^{2-\frac{3}{q}} \|\nabla \rho\|_{L^q}\\
        \le& C \bar\rho \left( \bar \rho^{\gamma-\alpha-1}\|\nabla \rho \|_{L^q} +\|\nabla u\|_{L^2} \right) \|\nabla \rho \|_{L^q} \\
        &+ C \bar\rho^{-\alpha+1} \|\rho u_t\|_{L^q} \|\nabla \rho \|_{L^q}  + C \bar \rho^{1-\alpha} \|\nabla^2 d\|_{L^2}^{\frac{3}{q}}\|\nabla^3 d\|_{L^2}^{2-\frac{3}{q}} \|\nabla \rho\|_{L^q}\\
        \le & \frac12 \bar \rho^{\gamma-\alpha}   \|\nabla \rho\|_{L^q}^2  + C \bar\rho^{2-\alpha-\gamma} \|\rho u_t\|_{L^q}^2 + C \bar \rho^{2+\alpha-\gamma} \|\nabla  u\|_{L^2}^2 \\
        &+ C \bar \rho^{1-\alpha} \|\nabla^2 d\|_{L^2}^{\frac{3}{q}}\|\nabla^3 d\|_{L^2}^{2-\frac{3}{q}} \|\nabla \rho\|_{L^q}.
    \end{aligned}
    \end{equation}

    Next, it follows from \eqref{a1} and \eqref{e-p} that 
    \begin{equation}\label{ep1}
        \begin{aligned}
            \int_{0}^{T}  \|\rho u_t\|_{L^q}^2 dt \le & \int_{0}^{T}  \bar\rho^{\frac{5q-6}{2q}}\|\sqrt{\rho}u_t\|_{L^2}^{\frac{6-q}{q}} \|\nabla u_t\|_{L^2}^{\frac{3(q-2)}{q}} dt\\
            \le & C \bar\rho^{\frac{5q-6}{2q}} \left(\int_{0}^{T} \|\sqrt{\rho}u_t\|_{L^2}^2 dt\right)^{\frac{6-q}{2q}} \left( \int_{0}^{T}\|\nabla u_t\|_{L^2}^2 dt \right)^{\frac{3(q-2)}{2q}}\\
            \le & C \bar\rho^{\frac{5q-6}{2q}+ \alpha \frac{6-q}{2q} + (\alpha-1)  \frac{3(q-2)}{2q}} = C \bar\rho^{1+\alpha},
        \end{aligned}
    \end{equation}
    and that 
    \begin{equation}\label{ddd}
        \begin{aligned}
            \int_{0}^{T} \|\nabla^2 d\|_{L^2}^{\frac{3}{q}}\|\nabla^3 d\|_{L^2}^{2-\frac{3}{q}} dt \le C.
        \end{aligned}
    \end{equation}
    % \begin{equation}\label{ep2}
    %     \begin{aligned}
    %         \int_{0}^{T}  \|\nabla u\|_{L^\infty}^2 dt 
    %         \le & C \int_{0}^{T} \|\nabla u\|_{L^2}^{\theta} \|\nabla^2 u\|_{L^q}^{1-\theta} dt\\
    %         %\le& C \int_{0}^{T} \|\nabla u\|_{L^2}^{2 \theta} \left( \bar\rho^{-\alpha}\|\rho u_t\|_{L^q} + \bar\rho^{\frac{-\alpha+1}{\theta}}  \| u\|_{L^q}^{\frac{1}{\theta}} \|\nabla u\|_{L^2} + \bar\rho^{-\frac{1}{\theta}} \|\nabla \rho\|_{L^q}^{\frac{1}{\theta}} \|\nabla u\|_{L^2}  + \bar \rho^{\gamma-\alpha-1}\|\nabla \rho \|_{L^q}\right)^{2(1-\theta)} dt \\
    %         \le & C \int_{0}^{T}  \biggl( \|\nabla u\|_{L^2}^{2 } 
    %         + \bar\rho^{-2\alpha}\|\rho u_t\|_{L^q}^2 + \bar\rho^{\frac{-2\alpha+2}{\theta}}  \| u\|_{L^q}^{\frac{2}{\theta}} \|\nabla u\|_{L^2}^2 \\
    %         & +\bar\rho^{-\frac{2}{\theta}} \|\nabla \rho\|_{L^q}^{\frac{2}{\theta}} \|\nabla u\|_{L^2}^2  
    %         + \bar \rho^{2\gamma-2\alpha-2}\|\nabla \rho \|_{L^q}^2 + \bar \rho^{2-2\alpha} \|\nabla^2 d\|_{L^2}^{\frac{6}{q}}\|\nabla^3 d\|_{L^2}^{4-\frac{6}{q}}\biggr) dt\\
    %         \le & C \int_{0}^{T}  \left( \|\nabla u\|_{L^2}^{2 } 
    %         + \bar\rho^{-2\alpha}\|\rho u_t\|_{L^q}^2  + \bar \rho^{2\gamma-2\alpha-2}\|\nabla \rho \|_{L^q}^2 \right) dt \\
    %         \le & C  \bar\rho^{1-\alpha} 
    %         + C\bar \rho^{\gamma-\alpha-2 + b}\\
    %         \le & C \bar \rho^{\gamma-\alpha},
    %     \end{aligned}
    % \end{equation}
    % with $\theta$ defined in \eqref{theta}.

    Finally, applying the Gronwall inequality to \eqref{rho1} and using \eqref{ep1}, \eqref{ddd}, we have 
    \begin{equation}
    \begin{aligned}
        &\sup_{0\le t\le T} \|\nabla \rho\|_{L^q}^2 + \frac{1}{2} \bar \rho^{\gamma-\alpha}   \int_{0}^{T} \|\nabla \rho\|_{L^q}^2 dt \\
        %\le& \exp{\left\{\int_{0}^{T} C \bar \rho^{\alpha-\gamma} \|\nabla u\|_{L^\infty}^2 dt \right\}}\\
        \le & C\left(\|\nabla \rho_0\|_{L^q}^2 + C \bar\rho^{2-\alpha-\gamma} \int_{0}^{T}\|\rho u_t\|_{L^q}^2 dt + C \bar\rho^{2+\alpha-\gamma}  \int_{0}^{T} \| \nabla u\|_{L^2}^2 dt \right)\\
        & + C \bar \rho^{1-\alpha} \int_0^T \|\nabla^2 d\|_{L^2}^{\frac{3}{q}}\|\nabla^3 d\|_{L^2}^{2-\frac{3}{q}} \|\nabla \rho\|_{L^q} dt\\
        \le & C \left(\|\nabla \rho_0\|_{L^q}^2 + \bar\rho^{3-\gamma} +  \bar\rho^{1-\alpha + \frac{\beta}{2}}  \right)\\
        \le & C_4 \left(\|\nabla \rho_0\|_{L^q}^2 +  \bar\rho^{3-\gamma} \right)\\
        \le & N_1 \bar \rho^{\beta},
    \end{aligned}
    \end{equation}
    where 
    \begin{equation}\label{m3}
        N_1 \triangleq C_4 \left(\|\nabla \rho_0\|_{L^q}^2 + 1\right).
    \end{equation}

        Similarly, we have
    \begin{equation}
    \begin{aligned}
        &\frac{d}{dt} \|\nabla \rho\|_{L^2}^2 + \bar \rho^{\gamma-\alpha}   \|\nabla \rho\|_{L^2}^2\\
        \le & C \bar\rho^{2-\alpha-\gamma} \|\rho u_t\|_{L^2}^2 + C \bar \rho^{2+\alpha-\gamma} \|\nabla  u\|_{L^2}^2 + C \bar \rho^{1-\alpha} \|\nabla^2 d\|_{L^2}^{\frac{3}{2}}\|\nabla^3 d\|_{L^2}^{\frac{1}{2}} \|\nabla \rho\|_{L^2},
    \end{aligned}
    \end{equation}
    which together with \eqref{ddd} gives
    \begin{equation}
    \begin{aligned}
        &\sup_{0\le t\le T} \|\nabla \rho\|_{L^2}^2 +  \bar \rho^{\gamma-\alpha}   \int_{0}^{T} \|\nabla \rho\|_{L^2}^2 dt \\
        \le & C\left(\|\nabla \rho_0\|_{L^2}^2 + C \bar\rho^{2-\alpha-\gamma} \int_{0}^{T}\|\rho u_t\|_{L^2}^2 dt + C \bar\rho^{2+\alpha-\gamma}  \int_{0}^{T} \| \nabla u\|_{L^2}^2 dt \right)\\
        & + C \bar \rho^{1-\alpha} \int_0^T \|\nabla^2 d\|_{L^2}^{\frac{3}{2}}\|\nabla^3 d\|_{L^2}^{\frac{1}{2}} \|\nabla \rho\|_{L^2} dt\\
        \le & C \left(\|\nabla \rho_0\|_{L^q}^2 + \bar\rho^{3-\gamma} +  \bar\rho^{1-\alpha + \frac{\beta}{2}}  \right)\\
        \le & C_5 \left(\|\nabla \rho_0\|_{L^2}^2 +  \bar\rho^{3-\gamma} \right)\\
        \le & N_2 \bar \rho^{\beta},
    \end{aligned}
    \end{equation}
    where 
    \begin{equation}\label{m33}
        N_2 \triangleq C_5 \left(\|\nabla \rho_0\|_{L^2}^2 + 1\right).
    \end{equation}
    
    Finally,   
    \begin{equation}
    \begin{aligned}
        \mathcal{E}_{\rho,1}(T)=\sup_{t\in[0,T] }\|\rho-\bar\rho \|_{L^\infty}\le &C \sup_{t\in[0,T] }\|\rho-\bar\rho \|_{L^2}^{\frac{2(q-3) }{5q-6}} \|\nabla \rho \|_{L^q}^{\frac{3q }{5q-6}}\\
        \le& C \bar \rho^{\frac{2(q-3) }{5q-6} \frac{3-\gamma}{2}+ \frac{\beta}{2} \frac{3q }{5q-6} } \\
        \le& C_6 \bar \rho^{\frac{2(q-3) }{5q-6} + \frac{\beta}{2} \frac{3q }{5q-6} }\le \frac{\bar \rho}{4}, 
    \end{aligned}
    \end{equation}
    provided that 
    \begin{equation}\label{lambda}
        \bar \rho \ge \Lambda_4 \triangleq \max \left\{\Lambda_3, (4C_6)^{\frac{2(5q-6)}{3q(2-\beta)}} \right\}.
    \end{equation}
\end{proof}
\textbf{Proof of Prosition \ref{pr}}
    Proposition \ref{pr} is a direct consequence of Lemmas \ref{3d}, \ref{L_2} and \ref{L_1} after choosing
    \begin{equation}\label{delta}
        \delta < \min\left\{\frac{1}{2\sqrt{2 c_1}}, \frac{1}{9 c_2}, \frac{1}{2c_3}, \frac{1}{4c_4}\right\},
    \end{equation}
    with constants $c_1,c_2,c_3,c_4$ depending only on the Sobolev constants and elliptic constants,
    and 
    $$\Lambda_0\triangleq\Lambda_4.$$
    
\section{Proof of Theorem \ref{global}}
According to Theorem \ref{local}, there exists a $\Tilde{T}>0$ such that the compressible simplified Ericksen-Leslie system \eqref{ins}-\eqref{bc} has a unique local strong solution $(\rho, u, d)$ on $[0, \Tilde{T}]$. We use the a priori estimates, Proposition \ref{pr} to extend the local strong solution to all time.

Due to 
\begin{equation}
    \|\nabla \rho_0\|_{L^q}=\mathcal{E}_\rho(0)<3\mathcal{E}_\rho(0), \ 
    \|\nabla u_0\|_{L^2}^2=M<3M,\
    \|\nabla d_0\|_{L^3}=\mathcal{E}_d(0)<2\delta,
\end{equation}
and the local regularity results \eqref{l-r}, there exists a $T_1\in(0, \Tilde{T})$ such that
\begin{equation}
    \sup_{0\le t\le T_1} \|\nabla \rho\|_{L^q}\leq 3\mathcal{E}_\rho(0),\ 
    \sup_{0\le t\le T_1} \|\nabla u_0\|_{L^2}^2 \leq 3M,\
    \sup_{0\le t\le T_1} \|\nabla d_0\|_{L^3} \leq 2\delta.
\end{equation}
Set
\begin{equation}
    T^\ast=\sup\{T| (\rho, u, d)\ \mathrm{is}\ \mathrm{a}\ \mathrm{strong}\ \mathrm{solution}\ \mathrm{to}\ \eqref{ins}-\eqref{bc}\ \mathrm{on}\ [0,T]\},
\end{equation}
\begin{equation}
T_1^\ast=\sup\left\{T\bigg| 
\begin{array}{l}
(\rho, u, d)\ \mathrm{is}\ \mathrm{a}\ \mathrm{strong}\ \mathrm{solution}\ \mathrm{to}\ \eqref{ins}-\eqref{bc}\ \mathrm{on}\ [0,T],\\
\sup_{0\le t\le T_1} \|\nabla \rho\|_{L^q}\leq 3\mathcal{E}_\rho(0),\ 
\sup_{0\le t\le T_1} \|\nabla u_0\|_{L^2}^2 \leq 3M,\\
\sup_{0\le t\le T_1} \|\nabla d_0\|_{L^3} \leq 2\delta.
\end{array} \right\}.
\end{equation}
Then $T^\ast_1\geq T_1>0$. Recalling Proposition \ref{pr}, it’s easy to verify
\begin{equation}
    T^\ast=T^\ast_1.
\end{equation}
provided that $\bar\rho>\Lambda_0$ and $\|\nabla d_0\|_{L^3}\leq  \varepsilon_0$ as assumed. 

We claim that $T^\ast=\infty$. Otherwise, assume that
$T^\ast<\infty$. By virtue of Proposition \ref{pr}, for every $t\in[0, T^\ast)$, it holds that
\begin{equation}
     \|\nabla \rho\|_{L^q}\leq 2\mathcal{E}_\rho(0),\ 
     \|\nabla u_0\|_{L^2}^2 \leq 2M,\
     \|\nabla d_0\|_{L^3} \leq \delta,
\end{equation}
therefore we can extend the solution to $T^{\ast\ast}>T^\ast$ due to Lemma \ref{local}, which contradicts with the defination of $T^\ast$. Hence we finish the proof of Theorem \ref{global}.

\section*{Conflict-of-interest statement}
All authors declare that they have no conflicts of interest.

\section*{Data Availability}
No data were used for the research described in the article.

\section*{Acknowledgments}
Y. Mei is supported by the National Natural Science Foundation of China No. 12101496 and 12371227.	
% J.-X. Li was supported in part by Zheng Ge Ru Foundation, 
% %Hong Kong RGC Earmarked Research Grants CUHK-14301421, CUHK-14300819, CUHK-14302819, CUHK-14300917, the key project of NSFC (Grant No. 12131010)  
% the Shun Hing Education and Charity Fund. 
R. Zhang is supported by the National Natural Science Foundation of China No. 12401279. 
%Part of this work was done when Y. Mei and R. Zhang were visiting the Institute of Mathematical Sciences at the Chinese University of Hong Kong. They would like to thank the institute for its hospitality.

\normalem
\bibliographystyle{siam}
\bibliography{ref}

\end{document}